\documentclass[11pt, notitlepage]{article}
\usepackage{amssymb,amsmath,comment}
\catcode`\@=11 \@addtoreset{equation}{section}
\def\thesection{\arabic{section}}

\def\theequation{\thesection.\arabic{equation}}
\catcode`\@=12
\usepackage{enumitem}
\usepackage{colortbl}%
\usepackage{a4wide}
\usepackage{parskip}
\usepackage{hyperref}

\newcommand{\ds} {\displaystyle}
\newcommand{\e}{\epsilon}

\newcommand{\al} {\alpha}
\newcommand{\ba} {\beta}

\newcommand{\de} {\delta}

\newcommand{\Om} {\Omega}
\newcommand{\ra} {\rightarrow}

\newcommand{\De} {\Delta}
\newcommand{\la} {\lambda}

\newcommand{\noi} {\noindent}
\newcommand{\na} {\nabla}

\newcommand{\mb} {\mathbb}
\newcommand{\mc} {\mathcal}

\newcommand{\ld} {\langle}
\newcommand{\rd} {\rangle}
\newcommand{\I}{\int\limits_}

\usepackage[all]{xy}
\catcode`\@=11
\def\theequation{\@arabic{\c@section}.\@arabic{\c@equation}}
\catcode`\@=12

\def\QED{\hfill {$\square$}\goodbreak \medskip}
\newtheorem{Theorem}{Theorem}[section]
\newtheorem{Lemma}[Theorem]{Lemma}
\newtheorem{Proposition}[Theorem]{Proposition}
\newtheorem{Corollary}[Theorem]{Corollary}
\newtheorem{Remark}[Theorem]{Remark}
\newtheorem{Definition}[Theorem]{Definition}

\def\XXint#1#2#3{{\setbox0=\hbox{$#1{#2#3}{\int}$ }
		\vcenter{\hbox{$#2#3$ }}\kern-.6\wd0}}

\begin{document}
	{\vspace{0.01in}
		\title
		{Choquard equation involving mixed local and nonlocal operators}
		\author{ {\bf G.C. Anthal\footnote{	Department of Mathematics, Indian Institute of Technology, Delhi, Hauz Khas, New Delhi-110016, India. e-mail: Gurdevanthal92@gmail.com}\;, J. Giacomoni\footnote{ 	LMAP (UMR E2S UPPA CNRS 5142) Bat. IPRA, Avenue de l'Universit\'{e}, 64013 Pau, France. 
					email: jacques.giacomoni@univ-pau.fr}  \;and  K. Sreenadh\footnote{Department of Mathematics, Indian Institute of Technology, Delhi, Hauz Khas, New Delhi-110016, India.	e-mail: sreenadh@maths.iitd.ac.in}}}
		\date{}
		
		\maketitle
		
	\begin{abstract}
In this article, we study an elliptic problem involving an operator of mixed order with both local and nonlocal aspects and in the presence of critical nonlinearity of Hartree type. To this end, we first investigate the corresponding Hardy-Littlewood-Sobolev inequality and detect the optimal constant. Using variational methods and a Poho\v{z}aev identity we then show the existence and nonexistence results for the corresponding subcritical perturbation problem.  \\
\noi \textbf{Key words:} Local-nonlocal operators, critical Choquard nonlinearity, Hardy-Littlewood-Sobolev inequality, existence reusults, variational methods, Poho\v{z}aev identity.  \\

\noi \textit{2020 Mathematics Subject Classification:} 35A01, 35A15, 35B33, 35B65. \\
	\end{abstract}	
\section{Introduction}\label{I}
 The aim of the present article is to study the following problem consisting of combination of local and nonlocal operators along with critical Choquard nonlinearity
 \begin{align}\label{e1.1}
 	\begin{cases}	\mc L u = \left(\ds\I{\Om}\frac{|u(y)|^{2_\mu^\ast}}{|x-y|^\mu}\right)|u|^{2_\mu^\ast-2}u +\la u^p~\text{in}~\Om,\\
 		u \equiv 0~\text{in}~\mb R^n \setminus \Om,~u\geq0~\text{in}~\Om,
 	\end{cases}
 \end{align}
 where $\Om$ is a bounded domain of $\mb R^n$ with $C^{1,1}$ boundary $\partial \Om$, $\la $ is a real parameter, $p \in [1,2^*-1)$, $n \geq 3$, $0 <\mu <n$, $2_\mu^\ast =(2n-\mu)/(n-2)$ and $2^* = 2n/(n-2)$. The mixed operator $\mc L$ in \eqref{e1.1} is given by 
 \begin{equation}\label{ecl}
 	\mc L =-\De +(-\De)^s~\text{for some}~ s \in (0,1).
 \end{equation}
 The word "mixed" refers to the type of the operator combining both local and nonlocal features and to the differential order of the operator. The operator $\mc  L$ is obtained by the superposition of the classical Laplacian $(-\De)$ and the fractional Laplacian $(-\De)^s$, which for a fixed parameter $s \in (0,1)$, is defined by 
 \begin{align*}
 	(-\De)^su =C(n,s)P.V. \I{\mb R^n} \frac{u(x)-u(y)}{|x-y|^{n+2s}}dy.
 \end{align*}
 The term $"P.V"$ stands for the Cauchy's principal value and $C(n,s)$ is a normalizing constant, whose explicit expression is given by
 \begin{equation*}
 	C(n,s)=\left(\I{\mb R^n}\frac{1-\cos(z_1)}{|z|^{n+2s}}dz\right)^{-1}.
 \end{equation*}
 The study of the mixed operators of the form $\mc L$ in \eqref{ecl} is motivated by the wide range of applications. Indeed these operators arise naturally in the applied sciences, to study the role of the impact caused by a local and a nonlocal change in a physical phenomenon. These operators model diffusion patterns with different time scales (loosely speaking, the higher order operator leading the diffusion for small scales times and the lower order operator becoming predominant for large times) and they arise for instance in bi-modal power law distribution processes, see \cite{PV1}. Further applications arise in the theory of optimal searching, biomathematics and animal forging, see \cite{DLV, DV} and the references therein. See also \cite{KLS, MV, PV} and the references therein for further applications.
 
 \noindent Due to these applications and mathematical interest, the study of elliptic problems involving mixed type of operators having both local and nonlocal features is attracting a lot of attention. The current research has specifically focused on several problems in the existence and regularity theory. In the following, we present a short literature review regarding existence and regularity of solutions of the problem of the type
 \begin{equation*}
 	-\De_q u +(-\De)_q^s u =f ~\text{in}~\Om,
 \end{equation*}
where $\Om \subset \mb R^n$ is some domain, $q \in (1,\infty)$, $s \in (0,1)$, $-\De_q$ and $(-\De)_q^s$ are the usual $q-$Laplacian and fractional $q-$Laplacian operators.
 
 \noindent In the linear case $q=2$, to study the structural results like existence of weak solutions, strong maximum principle, local boundedness, interior Sobolev and Lipschitz regularity, along with various other qualitative properties of solutions we refer to \cite{AC, BCCI, BDVV2}. We also refer to \cite{BDVV1, BDVV4, DLV, DV} and references therein for other issues as symmetry and properties associated to the first eigenvalue.
 In the nonlinear setting $q\ne 2$, for $f=0$, Garain and Kinnunen \cite{GK} obtained the regularity results for weak solutions in terms of local boundedness, Harnack estimates, local H\"{o}lder continuity and semicontinuity. For the inhomogenous case, boundedness and strong maximum principle has been established in \cite{BMV} (see also \cite{BDVV3}). The question of H\"{o}lder regularity was investigated by De Filipps-Mingione in \cite{FM} for a large class of mixed local and nonlocal  operators. Under some suitable assumptions, the authors prove almost Lipschitz local continuity and local H\"oder regularity of the gradient (see Theorem 3, 6 and 7 respectively there).
  We would also like to mention the work \cite{GU} where the case of singular nonlinearity is studied. Also to study regularity theory for nonhomogenous growth fractional problems, we refer to the work of Giacomoni-Kumar-Sreenadh \cite{GKS}.
 
 \noindent The stirring motivation for the present paper comes from the study of nonlinear problems with critical exponents. These problems are usually modelled as
 \begin{align}\label{ecp}
 	\mc J u = \mc W(u) +\la \mc Y(u) \in \Om,
 \end{align} 
 with some appropriate conditions on $u$. Here $\Om \subset \mb R^n$ is a domain, $\mc J$ is some operator (local/nonlocal/mixed), $\mc W$ is a critical nonlinearity (notion of criticality changes from problem to problem and relies on the limit of compactness inherited by the problem), $\mc Y$ is  subcritical nonlinearity and $\la \in \mb R$ is a parameter. In these problems mostly we address the issue of existence and multiplicity of the solutions with respect to the parameter $\la$. In the following, we present a brief literature survey of the problems of the type \eqref{ecp}. 
 \begin{itemize}
 	\item \textit{ Local case i.e. when $\mc J =-\De$}: The seminal breakthrough in this case was the work of Brezis and Nirenberg \cite{BN}. Here the authors studied the problem of the type
 	\begin{equation}\label{ebn}
 		\left\{-\De u =|u|^{2^*-2}u +\la u^q~\text{in}~\Om,~u=0~\text{on}~\partial \Om,~u>0~\text{in}~\Om \right.
 	\end{equation}
 	where $\Om \subset \mb R^n$ is a bounded domain and $2^* =2n/(n-2)$ is the critical exponent for the embedding of $H_0^1(\Om)$ to $L^p(\Om)$. When $\la=0$ and $\Om$ is starshaped with respect to some point $x \in \Om$, nonexistence of weak solutions holds. For some suitable range of $\la$, the authors investigate the existence of weak solutions. By developing some skill full techniques in estimating the minimax level, the authors were able to prove the existence of nontrivial solutions for the problem \eqref{ebn} under a linear and subcritical superlinear perturbation. Later in \cite{ABC} Ambrosetti et. al. investigated the problem \eqref{ebn} with the sublinear perturbation  and established the existence and multiplicity results depending on the parameter $\la$. We also refer to \cite{GP} to study the critical exponent problem for the
 	$p$-Laplacian operator.\\
 	The case of critical Choquard nonlinearity namely $ \mc W = \left(\ds\I{\Om}\frac{|u(y)|^{2_\mu^\ast}}{|x-y|^\mu}\right)|u|^{2_\mu^\ast-2}u $ with linear and sublinear or superlinear perturbations were studied by Gao and Yang in their works \cite{GY1} and \cite{GY} respectively and obtain the existence and multiplicity of solutions depending on some range of 
 	$\la$. Here the critical exponent $2_\mu^\ast$ is in accordance with well known Hardy-Littlewood-Sobolev inequality which states as under:
 	\begin{Proposition}\label{hls}
 		\textbf{Hardy-Littlewood-Sobolev inequality} Let $r, q>1$ and $0<\mu<n$ with $1/r+1/q+\mu/n=2$, $g\in L^{r}(\mb R^n), h\in L^q(\mb R^n)$. Then, there exist a sharp constant $C(r,q,n,\mu)$ independent of $g$ and $h$ such  that 
 		\begin{equation}\label{hlse}
 			\int\limits_{\mb R^n}\int\limits_{\mb R^n}\frac{g(x)h(y)}{|x-y|^{\mu}}dxdy \leq C(r,q,n,\mu) |g|_r|h|_q.
 		\end{equation}                             
 	\end{Proposition}
 	In particular, let $g = h = |u|^t$ 
 	then by Hardy-Littlewood-Sobolev inequality we see that,
 	$$                               \int\limits_{\mb R^n}\int\limits_{\mb R^n}\frac{|u(x)^t|u(y)|^t}{|x-y|^{\mu}}dxdy$$
 	is well defined if $|u|^t \in L^\nu(\mb R^n)$ with $\nu =\frac{2n}{2n-\mu}>1$.
 	Thus, from Sobolev embedding theorems, we must have
 	\begin{equation*}
 		\frac{2n-\mu}{n} \leq t \leq \frac{2n-\mu}{n-2}.
 	\end{equation*}
 	From this, for $u \in H^1(\mb R^n) $ we have
 	$$    \left(\int\limits_{\mb R^n}\int\limits_{\mb R^n}\frac{|u(x)|^{2_\mu^\ast}|u(y)|^{2_\mu^\ast}}{|x-y|^\mu}dxdy \right)^\frac{1}{2_\mu^\ast} \leq C(n,\mu)^\frac{1}{2_\mu^\ast}  |u|_{2^*}^2  .                                     $$
 	We denote $S_{H,L,C}$ to denote the best constant associated to
 	\begin{align}\label{ebc}	S_{H,L,C}=\inf\limits_{u \in C_0^\infty(\mb R^n)\setminus \{0\}} \frac{\|\nabla u\|_{L^2(\mb R^n)}^2}{\left(\I{\mb R^n}\I{\mb R^n}\frac{|u(x)|^{2_\mu^\ast}|u(y)|^{2_\mu^\ast}}{|x-y|^\mu}dxdy\right)^\frac{1}{2_\mu^\ast}}.
 	\end{align}
 	\item \textit{Nonlocal case i.e., when $\mc J = (-\De)^s$, for some $s \in (0,1)$}: Servadei and Valdinoci  studied in \cite{SV2} the critical problems in the case of linear perturbation with $\mc W(u) =|u|^{2_s^*-2}u$ and  $2_s^\ast =2n/(n-2s)$, the critical exponent in the fractional Sobolev inequality. The cases of sublinear and superlinear perturbations were handled by Barrios et. al. in their work \cite{BCSS}. In both of these works the authors showed the existence and multiplicity of solutions for some range of $\la$. 
Also the critical problem with fractional $p$-Laplacian was studied in \cite{MPSY}.\\
 	The case of critical Choquard nonlinearity  i.e., when $\mc W = \left(\ds\I{\Om}\frac{|u(y)|^{2_{\mu,s}^\ast}}{|x-y|^\mu}\right)|u|^{2_{\mu,s}^\ast-2}u $  with linear perturbation was handled by Mukherjee and Sreenadh \cite{MS}. Here the authors proved the following existence results depending up on different values of $n$ and $\la$:
 	\begin{enumerate}
 		\item For $n \geq 4s$, $s \in (0,1)$, the problem \eqref{ecp} has a nontrivial solution for every $\la>0$ such that $\la$ is not an eigenvalue of $(-\De)^s$ with homogenous Dirichlet boundary condition on $\mb R^n\setminus\Om$.
 		\item For $s \in (0,1)$ and $2s<n<4s$, then there exist $\bar{\la} >0$ such that for any $\la > \bar{\la}$ different from the eigenvalues of $(-\De)^s$ with homogenous Dirichlet boundary condition on $\mb R^n \setminus \Om$, \eqref{ecp} has a nontrivial weak solution.
 	\end{enumerate}
 	\item \textit{ Mixed operator case i.e., when $\mc J =\mc L$}: The critical problems involving operators of mixed type are very little explored. In this regard we can quote \cite[Biagi et al]{BDVV}. Here the authors first investigated the mixed Sobolev inequality and showed that the Sobolev constant in the mixed case is actually equal to the classical Sobolev constant. They then used this information to obtain the existence results both for linear and superlinear subcritical perturbation.
 \end{itemize}
 Motivated by the above discussion, in the present work we considered the problem \eqref{e1.1}. As far as we know there is no result involving the mixed local and nonlocal operator with a critical nonlocal term as the Choquard nonlinearity. We aim to bridge this gap in the present paper. As we know the study of problems with Choquard nonlinearity is linked to the Hardy-Littlewood-Sobolev inequality, we first investigated the  inequality for the mixed operator case. More precisely, we consider a fractional exponent $s \in (0,1)$, an open set $\Om \subset \mb R^n$, not necessarily bounded or connected, and all functions $u:\mb R^n \ra \mb R$ which vanish outside $\Om$, accounting for a mixed Hardy-Littlewood-Sobolev inequality of the type
 \begin{align*}
 	S_{H,L,M}(\Om)\I\Om\I\Om\frac{|u(x)|^{2_\mu^\ast}|u(y)|^{2_\mu^\ast}}{|x-y|^\mu}dxdy \leq \| \na u\|_{L^2(\mb R^n)}^2 + \frac{C(n,s)}{2}\I{\mb R^n}\I{\mb R^n}\frac{|u(x)-u(y)|^2}{|x-y|^{n+2s}}dxdy.
 \end{align*}
 Here above, the constant $S_{H,L,M}(\Om)$ is taken to be the largest possible one for which such an inequality holds true. Using \eqref{ebc} it is easy to see that 
 \begin{equation*}
 	S_{H,L,M}(\Om) \geq S_{H,L,C}.
 \end{equation*}
 In principle, one may suspect in fact a strict inequality occurs (because, for instance, $S_{H,L,C}$ is independent of $\Om$, as well as of $s$), but this is not the case according to the following result we prove:
 \begin{Theorem}\label{t1.1}
 	Let $s \in (0,1)$ and $\Om \subset \mb R^n$ be an arbitrary open set. Then we have $S_{H,L,M}(\Om)=S_{H,L,C}$.
 \end{Theorem}
 Our next result answers the question whether or not the optimal constant $S_{H,L,M}(\Om)$ is achieved. It states as under:
 \begin{Theorem}\label{t1.2}
 	Let $\Om \subseteq \mb R^n$ be an arbitrary open set. Then the optimal constant $S_{H,L,M}(\Om)$ is never achieved. 
 \end{Theorem}
  Next we address the problem \eqref{e1.1}. We start with the notion of weak solution for \eqref{e1.1}.
  \begin{Definition}\label{d4.1}
  	We say that a function $u \in \Pi(\Om)$ (definition of $\Pi(\Om)$ is  given in Section \ref{P}) is a solution of \eqref{e1.1} if for all $\varphi \in \Pi(\Om)$, we have
  	\begin{align}\nonumber\label{4.1}
  		\I\Om \na u\na \varphi dx +\frac{C(n,s)}{2} \I{\mb R^n}\I{\mb R^n} \frac{(u(x)-u(y))(\varphi(x)-\varphi(y))}{|x-y|^{N+2s}}dxdy =&\I\Om\I\Om \frac{|u(y)|^{2_\mu^\ast}|u(x)|^{2_\mu^\ast-2}u(x)\varphi(x)}{|x-y|^\mu}\\
  		&+\I\Om \la u^p \varphi dx.
  	\end{align}
  \end{Definition}
   According to  remark \ref{remarkdef}, we state our first regularity result below:
 \begin{Theorem}\label{trr}
 	Let $ u \in \Pi(\Om)$ be a weak solution of \eqref{e1.1}. Then $u \in L^\infty(\mb R^n) \cap W^{2,q}(\Om)$ for any $q \in (1,\infty)$ if $s \in (0,1/2]$ and $q \in \left(0,\frac{n}{2s-1}\right)$ if $s \in (1/2,1)$. Therefore $u \in C^{1,\nu}(\bar{\Om})$ for every $\nu \in (0,1)$ if $s \in (0,1/2]$ and for $\nu \in (0,2-2s)$ if $s \in(1/2,1)$. Further if we suppose that $\Om$ is of class $C^{2,\al}$, $\al \in (0,1)$, then $u \in C^{2,\beta}(\bar{\Om})$ for any $\ba$ such that $0<\beta<\min\{s,\alpha, n-\mu\}$.
 \end{Theorem}
	
With the help of the regularity result, we immediately obtain the following maximum principle:
\begin{Proposition}\label{MP}
	Let $\al \in (0,1)$ be as in Theorem \ref{trr} and suppose that $\partial \Om$ is of class $C^{2,\al}$. Assume that $u \in C^{1,\al}(\bar{\Om}) \cap C_{\text{loc}}^2(\Om)$ be a nontrivial nonnegative solution of $\eqref{e1.1}$. Then $u >0$ in $\Om$.
\end{Proposition}
 Our next result is the following extension of Poho\v{z}aev identity:
 \begin{Proposition}\label{PPI}
 	Let $s \in (0,1)$, $\Om$ be a $C^{2,\al}$ domain with $\al$ as in Theorem \ref{trr} and $u \in \Pi(\Om)$ solves \eqref{e1.1}, then
 	\begin{align}\label{pi}\nonumber
 		\frac{\mu-2n}{2\cdot2_\mu^\ast}\I\Om\I\Om&\frac{|u(x)|^{2_\mu^\ast}|u(y)|^{2_\mu^\ast}}{|x-y|^\mu}dxdy-\frac{\la n}{p+1}\I\Om |u|^{p+1}dx	\\ 
 		=&\frac{2-n}{2}\I\Om |\na u|^2dx
 		+\frac{2s-n}{2}\I\Om u(-\De)^su dx\\ \nonumber
 		&-\frac{1}{2}\I{\partial\Om}\left(\frac{\partial u}{\partial \nu}\right)^2\nu(x).x d\sigma-\frac{\Gamma(1+s)^2}{2}\I{\partial\Om}\left(\frac{u}{\de^s}\right)^2(\nu(x).x)d\sigma.
 	\end{align}
 \end{Proposition} 
Based on the above Poho\v{z}aev identity, we have the following nonexistence result of nonnegative and nontrivial solutions on bounded and strictly star-shaped domains:
 \begin{Theorem}\label{t1.4}
 	Suppose that $\Om$ is a strictly star shaped domain with $C^{2,\al}$ boundary with $\al$ as in Theorem \ref{trr}. Then \eqref{e1.1} cannot have a nontrivial solution provided the following holds
 	\begin{equation}\label{epn}
 		-\la	\left(n\left(\frac{1}{p+1}-\frac12\right) +1 \right)\geq 0.
 	\end{equation}
 \end{Theorem}
The following Corollary is immediate:
\begin{Corollary}
	Let $\Om$ be a strictly star shaped domain (with respect to the origin) with $C^{1,1}$ boundary. Then \eqref{e1.1} cannot have a nontrivial solution provided the following holds
	\begin{align*}
		p\geq \frac{n+2}{n-2}~\text{and}~ \la \geq 0~\text{or}~	p< \frac{n+2}{n-2}~\text{and}~ \la \leq 0.
	\end{align*}
\end{Corollary}
We want to remark that if we drop $C^{2,\al}$ regularity of $\partial \Om$, it is still possible to get $C_{\text{loc}}^{2,\al}$ regularity of $u$ but for restricted range of $s$ as it is stated in the result below: 
\begin{Theorem}\label{t4.3}
	Let $s \in (0,3/4)$ and $\Om \subset \mb R^n$ be a bounded domain with $C^{1,1}$ boundary. Then any nontrivial solution $u$ of \eqref{e1.1} is in $C_{\text{loc}}^{2,\al}(\Om)$ for any $\al \in (0,1)$ if $s\leq \frac{1}{2}$ and for any $\alpha\in (0, 3-4s)$ if $s>\frac{1}{2}$. Furthermore, the Poho\v{z}aev identity \eqref{pi} and Theorem \ref{MP}  hold.
\end{Theorem}
 Finally, we study the existence theory for the problem \eqref{e1.1} with respect to the parameter $\la$. Let us briefly recall the main strategy used in \cite{GY1} in the case $p=1$.
 
 \noindent Due to lack of compactness caused by the critical exponent, an idea borrowed from \cite{GY1} consists in proving that the Palais Smale level lies in the range $c \in \left(0, \ds\frac{n+2-\mu}{4n-2\mu}S_{H,L,C}^\frac{2n-\mu}{n+2-\mu}\right)$. In order to do so, the crucial step is to show that 
 \begin{equation*}
 	S(\la):=\inf\left\{\|u\|_2^2 -\la |u|_2^2: u\in H_0^1(\Om)~\text{and}~\|u\|_{HL}=1\right\} <S_{H,L,C}.
 \end{equation*}
Now, the strict inequality $S(\la)<S_{H,L,C}$ is obtained in \cite{GY1} via the following approach:\\
First, taking into account that the minimizers in the Hardy-Littlewood-Sobolev inequality are given by 
Aubin-Talenti functions, one consider the function
\begin{equation*}
	u_\e = \frac{\psi}{(\e^2 +|x|^2)^{(n-2)/2}},~\e>0.
\end{equation*}
 where $\psi \in C^{\infty}_{c}(\Om)$. 
Then, using $u_\e$ as a competitor function, one gets (at least for $n \geq 5$) that
\begin{equation*}
	\frac{\|u_\e\|_2^2-\la|u|_2^2}{\|u\|_{HL}^2}=S_{H,L,C}-d\la e^2+O(\e^{n-2})~\text{as}~\e \ra 0^+,
\end{equation*}
where $d>0$ is a suitable constant. From this, choosing $\e$ sufficiently small, one immediately conclude that $S(\la)<S_{H,L,C}$. A similar approach works for $1<p<2^*-1$ as well, and it has also been used in the nonlocal framework \cite{MS}.

\noindent In our mixed setting the situation is quite different. In fact, in trying to repeat the above argument as it obtained in \cite{SV2}, one is led to consider the following minimization problem 
\begin{equation*}
		S_{H,L}(\la):=\inf\left\{\mc G(u)^2 -\la |u|_2^2: u\in H_0^1(\Om)~\text{and}~\|u\|_{HL}=1\right\}
\end{equation*}
and to prove that
\begin{equation}\label{ess}
 S_{H,L}(\la)<S_{H,L,C}.
 \end{equation} Now since, we know from Theorem \ref{t1.1} that $S_{H,L,M}(\Om)=S_{H,L,C}$, in order to prove \eqref{ess}, it is natural to consider the test function $u_\e$ defined above. However, the presence of the nonlocal term $[u_\e]_s^2$ gives in this case 
\begin{equation*}
		\frac{\mc G(u_\e)^2-\la|u|_2^2}{\|u\|_{HL}^2}=S_{H,L,C}+O(\e^{2-2s})-d\la e^2+O(\e^{n-2})~\text{as}~\e \ra 0^+,
\end{equation*}
and the term $O(\e^{2-2s})$ is not negligible when $\e \ra 0^+$. 
In order to get the existence of solutions we borrowed ideas from \cite{BDVV} and study the properties of the map $\la \mapsto S_{H,L}(\la)$ to obtain the inequality \eqref{ess}.

 \noi All that being said, in the linear case $p=1$ we obtain that the problem \eqref{e1.1} does not admit solution both in the range of "small" and "large" values of $\la$, but it does possess solutions for an "intermediate" regime values of $\la$. More precisely, denoting by $\la_{1,s}$ the smallest Dirichlet eigenvalues of $(-\De)^s$ in a bounded open set $\Om$, and by $\la_1$ be the smallest Dirichlet eigenvalue of $\mc L$ in $\Om$, we have the following result:
 \begin{Theorem}\label{tlc}
 	Let $\Om \subset \mb R^n$ be an open and bounded set and $p=1$. There exists $\la^* \in [\la_{1,s},\la_1)$ such that the problem \eqref{e1.1} possesses at least one solution if
 	\begin{equation*}
 		\la*<\la <\la_1.
 	\end{equation*}
 	Moreover the following nonexistence results also hold:
 	\begin{enumerate}
 		\item there do no exist solutions to problem \eqref{e1.1} if $\la \geq \la_1$;
 		\item if $0<\la<\la_{1,s}$, then there do not exist  solutions to \eqref{e1.1} in $\mc B$, where
 		\begin{align*}
 			\mc B:=\{u \in L^{2^*}(\mb R^n):\|u\|_{HL}\leq S_{H,L,C}^\frac{n-2}{4-\mu}\}.
 		\end{align*}
 	\end{enumerate}
 \end{Theorem}
 Regarding the case of superlinear subcritical perturbation, we can follow the variational argument in \cite{GY1} to prove finally the following existence result:
 \begin{Theorem}\label{t4.1}
 	Assume that $1 <p<2^*-1$, $n \geq 3$ and $0<\mu <n$. Then problem \eqref{e1.1} has at least one nontrivial solution provided that either
 	\begin{enumerate}
 		\item $n> \max\left\{\min\left\{\frac{2(p+3)}{p+1},2+\frac{\mu}{p+1},2\left(1+\frac{2-2s}{p-1}\right)\right\},\frac{2(p+1)}{p}\right\}$ and $\la >0$, or
 		\item $n\leq \max\left\{\min\left\{\frac{2(p+3)}{p+1},2+\frac{\mu}{p+1},2\left(1+\frac{2-2s}{p-1}\right)\right\},\frac{2(p+1)}{p}\right\}$ and $\la$ is sufficiently large.
 	\end{enumerate}	
 \end{Theorem}
 	
 \textbf{Plan of the paper:} The rest of the paper is organized as follows. In Section \ref{P} we collect the preliminary material needed to setup the appropriate functional space. In Section \ref{HLS}, we focused on the mixed order Hardy-Littlewood-Sobolev inequality and prove Theorems \ref{t1.1} and \ref{t1.2}. Finally in Section \ref{SCP}, we present the analysis of critical problem. This section is divided into three subsections. In Subsection \ref{SRP}, we first focus on regularity of solutions and prove Theorem \ref{trr}. Then with the help of regularity result, we establish the maximum principle and prove Proposition \ref{MP}. We conclude this subsection by first developing the Poho\v{z}aev identity for \eqref{e1.1} (see Proposition \ref{PPI}) and then use this identity to prove Theorem \ref{t1.4}. Next in Subsection \ref{SLC}, we consider the case $p=1$ and complete the proof of Theorem \ref{tlc}. Finally, in Subsection \ref{SSC}, we consider the superlinear and subcritical case and prove Theorem \ref{t4.1}.
 
 \textbf{Notations:} Throughout the paper, we will use the following notations:
 \begin{itemize}
 	\item We denote positive constants by $M, M_1, M_2, \cdots$.
 	\item We denote the standard norm on $L^p(\mb R^n)$ by $|\cdot|_p$. 
 	\item Let $V$ be a real Hilbert space and $\Phi: V \ra \mb R$ a functional of class $C^1$. We say that $\{u_k\}_{k\in \mb N} \subset V$ is a Palais-Smale sequence at level $c$, usually denoted by $(PS)_c$, for $\Phi$ if $\{u_k\}_{k \in \mb N}$ satisfies
 	\begin{equation*}
 		\Phi(u_k) \ra c~\text{and}~\Phi'(u_k) \ra 0,~\text{as}~k \ra \infty.
 	\end{equation*}
 	Moreover, $\Phi$ satisfies the $(PS)_c$ condition at $c$, if any $(PS)$ sequence at $c$ possesses a convergent subsequence.
 \item We shall also use the following notation
 	\begin{align*}
 		\|u\|_{HL} =\left(\I{\mb R^n} \I{\mb R^n} \frac{|u(x)|^{2_\mu^\ast}|u(y)|^{2_\mu^\ast}}{|x-y|^\mu}dxdy\right)^\frac{1}{2 .2_\mu^\ast}.
 		u \in \Pi(\Om), 
 	\end{align*}	
 The definition of $\Pi(\Om)$ is given in Section \ref{P}.
 \end{itemize}
\section{Preliminaries}\label{P}
		In this section we shall give the functional settings required to study the problem \eqref{e1.1} and also collect the notations and preliminary results required in the rest of the paper.\\
		Let $s \in (0,1)$. For a measurable function $u:\mb R^n \ra \mb R$, we define
		\begin{equation*}
			[u]_s =\left(\frac{C(n,s)}{2} \I{\mb R^n}\I{\mb R^n}\frac{|u(x)-u(y)|^2}{|x-y|^{n +2s}}dxdy\right)^\frac12,
		\end{equation*}
	the so-called Gagliardo seminorm of $u$ of order $s$.	\\
	Let $ \Om \subset 
	\mb R^n$ be an arbitrary non-empty open set, not necessarily bounded. We define the space $\Pi (\Om)$ as the completion of $C_0^\infty (\Om)$ with respect to the global norm
	\begin{align*}
			\mc G(u):= \left(\| u\|^2+[u]_s^2\right)^\frac12, ~ u \in C_0^\infty(\Om),
		\end{align*}
	where we denote $\|u\|^2 =\ds\I{\mb R^n}|\na u|^2$.
	\begin{Remark}\label{remarkdef}
		A few remarks are in order:
		\begin{enumerate}
			\item The norm $\mc G(\cdot)$ is induced by the scalar product
			\begin{equation*}
				\ld u,v\rd_{\mc G}:= \I{\mb R^n} \na u\cdot\na v dx + \frac{C(n,s)}{2}\I{\mb R^n}\I{\mb R^n}\frac{(u(x)-u(y))(v(x)-v(y))}{|x-y|^{n+2s}}dxdy,
				\end{equation*}
			where $\cdot$ denotes the usual scalar product in the Euclidean space $\mb R^n$, and $\Pi(\Om)$ is a Hilbert space.
			\item Note that in the definition of $\mc G(\cdot)$ the $L^2$-norm of $\na u$ is considered on the whole of $\mb R^n$ in spite of $u \in C_0^\infty(\Om)$ (identically vanishes outside $\Om$). This is to point out that the elements in $\Pi(\Om)$ are functions defined on the entire space and not only on $\Om$. The benefit of having this global functional setting is that these functions can be globally approximated on $\mb R^n$ with respect to the norm $\mc G(\cdot)$ by smooth functions with support in $\Om$.\\
			In particular, when $\Om \ne \mb R^n$, we will see that this global definition of $\mb G(\cdot)$ implies that the functions in $\Pi(\Om)$ naturally satisfy the nonlocal Dirichlet condition prescribed in problem \eqref{e1.1}, that is,
			\begin{equation}\label{e1.2}
				u \equiv 0 ~\text{a.e. in}~\mb R^n\setminus\Om~\text{for every}~u \in \Pi(\Om).
			\end{equation} 
		In order to verify \eqref{e1.2}, we distinguish between two cases.\\
		\textbf{Case A}: When $\Om$ is bounded. In this case, we know (see \cite[Proposition 2.2]{DPV}) $H^1(\mb R^n)$ is continuously embedded into $H^s(\mb R^n)$ (with $s\in(0,1)$) $i.e.$ there exists a constant $k =k(s)>0$ such that, for every $u \in C_0^\infty(\Om)$ one has
		\begin{equation}\label{e2.2}
			[u]_s^2 \leq k(s)\|u\|_{H^1(\mb R^n)}^2=k(s)(\|u\|_{L^2(\mb R^n)}^2+\|\na u\|_{L^2(\mb R^n)}^2).
		\end{equation}
	This, together with the classical Poincar\'{e} inequality, implies that $\mc G(\cdot)$ and the full $H^1-norm$ in $\mb R^n$ are actually equivalent in the space $C^\infty_0(\Om)$, and hence
	\begin{align*}
		\Pi(\Om)=\overline{C_0^\infty(\Om)}^{\|\cdot\|_{H^1(\mb R^{n})}}=\{u \in H^1(\mb R^n): u\arrowvert_\Om\in H_0^1(\Om)~\text{and}~u\equiv 0 ~\text{a.e. in}~\mb R^n\setminus \Om\}.
	\end{align*}
\textbf{Case B}: When $\Om$ is unbounded. In this case, even if the embedding inequality \eqref{e2.2} holds, the Poincar\'e inequality does not hold in general. Hence, the norm $\mc G(\cdot)$ is no more equivalent to the full norm $H^1$-norm in $\mb R^n$, and $\Pi(\Om)$ is not a subspace of $H^1(\mb R^n)$.\\
On the other hand, by the classical Sobolev inequality we infer the existence of a constant $S >0$, independent of the set $\Om$, such that
\begin{equation}\label{e2.3}
	S \|u\|_{L^{2^\ast}(\mb R^n)}^2 \leq \|\na u\|_{L^2(\mb R^n)}^2 \leq \mc G(\cdot)^2~\text{for every}~u \in C_0^\infty(\Om),
\end{equation}
where $2^*=2n/(n-2)$ is the critical Sobolev exponent. Now \eqref{e2.3} implies that every Cauchy sequence in $C_0^\infty(\Om)$ with respect to the norm $\mc G(\cdot)$ is also a Cauchy sequence in the space $L^{2^*}(\mb R^n)$. As a consequence, since the functions in $C_0^\infty(\Om)$ identically vanish out of $\Om$, we obtain
\begin{align*}
	\Pi(\Om)=\{u \in L^{2^*}(\mb R^n): u\equiv 0~\text{a.e. in}~\mb R^n\setminus \Om,~\na u \in L^2(\mb R^n)~\text{and}~[u]_s <\infty\}.
\end{align*}
	\end{enumerate}
	\end{Remark}
\section{Mixed Hardy-Littlewood- Sobolev inequality}\label{HLS}		
	Here we are interested to find the sharp constant in the Hardy-Littlewood-Sobolev inequality defined by
	\begin{align}\label{mhc}
		S_{H,L,M}(\Om)=\inf\limits_{u \in C_0^\infty(\Om)\setminus \{0\}} \frac{\mc G(u)^2}{\left(\I{\mb R^n}\I{\mb R^n}\frac{|u(x)|^{2_\mu^\ast}|u(y)|^{2_\mu^\ast}}{|x-y|^\mu}dxdy\right)^\frac{1}{2_\mu^\ast}}.
	\end{align}	
Recall that for any $\Om \subset \mb R^n$, the best constant in the classical Hardy-Littlewood-Sobolev inequality is given as
	\begin{align}\label{chc}
		S_{H,L,C}(\Om)=\inf\limits_{u \in C_0^\infty(\Om)\setminus \{0\}} \frac{\|\nabla u\|_{L^2(\Om)}^2}{\left(\I\Om\I\Om\frac{|u(x)|^{2_\mu^\ast}|u(y)|^{2_\mu^\ast}}{|x-y|^\mu}dxdy\right)^\frac{1}{2_\mu^\ast}}.
	\end{align}
	Let us first recall properties of $S_{H,L,C}(\Om)$. For a proof of these properties we refer to \cite{GY1} 
	\begin{Remark}\label{r3.1}
	\begin{enumerate}
		\item For $N \geq 3$ and for every open set $\Om$ of $\mb R^n$, we have $S_{H,L,C}(\Om)=S_{H,L,C}$.
		\item $S_{H,L,C}(\Om)$ is never achieved except when $\Om =\mb R^n$.
		\item If $\Om =\mb R^n$, then $S_{H,L,C}$ is achieved by the family of functions
		\begin{equation*}
			\mc A =\{V_{t,x_0}=t^{\frac{2-n}{2} }\mc U ((x- x_0)/t):t>0,~x_0 \in \mb R^n\},
		\end{equation*}
	where $\mc U(y);= c(1+|y|^2)^\frac{2-n}{2}$. 
	\end{enumerate}	
\end{Remark}
We are now ready to prove Theorem \ref{t1.1}
	
\textbf{ Proof of Theorem \ref{t1.1}}:
	From the definition of $\mathcal{G}(u)$, the inequality $S_{H,L,M}(\Om) \geq S_{H,L,C}$ is clear. Next we prove the reverse inequality. First note that $S_{H,L,M}(\Om)$ is translation invariant and so we can assume without loss of generality that $ 0\in \Om$. Now let $r >0$ be such that $B_r(0) \subseteq \Om$. Then for any $u \in C_0^\infty(\Om)$, there exists $k_0(u) \in \mb N$ such that
	\begin{equation*}
		\text{supp} (u) \subseteq B_{kr}(0),~\text{for any}~k \geq k_0.
	\end{equation*}
Now setting $u_k:= k^\frac{n-2}{2}u(kx)$, for $k\geq k_0$ we see that
\begin{align*}
	\text{supp}(u_k) \subseteq B_r(0) \subseteq \Om.
\end{align*}
 Now by definition of $S_{H,L,M}(\Om)$ we have for every $k\geq k_0$ that
 \begin{align*}
 	S_{H,L,M}(\Om) \leq \frac{\mc G(u_k)^2}{\|u_k\|_{HL}^2} =\frac{\|\nabla u\|_{L^2(\mb R^n)}^2}{\|u\|_{HL}^2}+k^{2s-2}\frac{[u]_s^2}{\|u\|_{HL}^2}.
 \end{align*}
Let $k \ra \infty$ and using $0<s<1$, we obtain
\begin{equation*}
	S_{H,L,M}(\Om) \leq \frac{\|\nabla u\|_{L^2(\mb R^n)}^2}{\|u\|_{HL}^2}.
\end{equation*}
By the arbitrariness of $u \in C_0^\infty(\Om)$ and the fact that $S_{H,L,C}$ is independent of the set $\Om$, we finally infer that
\begin{equation*}
	S_{H,L,M}(\Om) \leq S_{H,L,C}.
\end{equation*}
Hence we have proved the required equality. \QED
We are now in the position of proving Theorem \ref{t1.2}.\\
\textbf{Proof of Theorem \ref{t1.2}}:
		Suppose on the contrary that there exists a nonzero function $v \in \Pi(\Om)$ such that $\|v\|_{HL}=1$ and
		\begin{equation*}
			\mc G(v)=S_{H,L,M}(\Om) =S_{H,L,C}.
		\end{equation*}
	Noting that $\na u \in L^2(\Om)$  , we infer that
	\begin{align*}
		S_{H,L,C}\leq \|v\|^2 \leq \| v\|^2 +[v]^2=S_{H,L,C},
	\end{align*}
which implies that $[v]_s=0$. As a consequence, the function $u_0$ must be constant in $\Om$, which contradicts the fact that $\|v\|_{HL}=1$. \QED
	\begin{Remark}
		Even if Theorem \ref{t1.2} show that the constant $S_{H,L,M}(\Om)=S_{H,L,C}$ is never achieved in the space $\Pi(\Om)$ (independently of the set $\Om$). In the particular case $\Om = \mb R^n$ we can prove that $S_{H,L,C}$ is achieved 'in the limit': more precisely, if $\mc A =\{V_{t,x_0}\}$ is as in Remark \eqref{r3.1}-$3$, we have
		\begin{equation*}
			\mc G(V_{t,x_0})^2 \ra S_{H,L,C}~\text{as}~t \ra \infty.
		\end{equation*}
	\end{Remark}
Indeed,
\begin{align*}
	|\mc U(y)| \leq M \min \{1,|y|^{2-n}\}~\text{and}~|\na \mc U(y)| \leq M \min \{|y|,|y|^{1-n}\},
\end{align*}
for some $M>0$. Therefore (up to renaming $M$ line after line),
\begin{align*}
	[\mc U(y)]_s^2 &= \I{ \mb R^n}\I{ \mb R^n}\frac{|\mc U(x+z)-\mc U(x)|^2}{|z|^{n+2s}}dxdz\\
	&\leq \I{\mb R^n}\I{B_1}\left| \I{0}^{1}\na \mc U (x+tz)\cdot z dt\right|^2 \frac{dxdz}{|z|^{n+2s}} +2\I{\mb R^n}\I{\mb R^n \setminus B_1}\left( |\mc U (x+z)|^2 +|\mc U (x)|^2\right)\frac{dxdz}{|z|^{n+2s}}\\
	&\leq M\I{\mb R^n}\I{B_1}\I{0}^{1}\min\{|x+tz|^2,|x+tz|^{2(1-n)}\}\frac{dxdzdt}{|z|^{n+2s-2}} +4\I{\mb R^n}\I{\mb R^n \setminus B_1} |\mc U(y)|^2 \frac{dydz}{|z|^{n+2s}}\\
	&\leq M \I{\mb R^n}\I{B_1}\I{0}^{1}\min\{|y|^2,|y|^{2(1-n)}\}\frac{dydzdt}{|z|^{n+2s-2}} +M \I{\mb R^n}\I{\mb R^n \setminus B_1} \min\{1,|y|^{2(2-n)}\}\frac{dzdy}{|z|^{n+2s}},
	\end{align*}
which is finite.\\
Hence, $\mc U \in \Pi(\mb R^n)$ and consequently $V_{t,x_0} \in \Pi(\mb R^n)$ for every $t>0$ and $x_0 \in R^n$. Moreover, recalling that
\begin{align*}
	V_{t,x_0}(x) =t^\frac{2-n}{2}\mc U\left(\frac{x-x_0}{t}\right)~\text{and}~\| V_{t,x_0}\|_{HL} =\|\mc U\|_{HL}=1,
\end{align*}
and by arguing as in the proof of Theorem \ref{t1.1} we have
\begin{equation*}
	\mc G (V_{t,x_0})^2 =\mc G (V_{t,0})^2 =\|\mc U\|^2 +t^{2-2s}[\mc U]_s^2. 
\end{equation*}
From this, since $\mc U =V_{1,0}$ is an optimal function in the classical Hardy-Littlewood- Sobolev inequality, by letting $t \ra 0$ we obtain
\begin{equation*}
	\mc G (V_{t,x_0})^2 \ra \|\mc U\|^2 =S_{H,L,C}.
	\end{equation*}
	\section{Study of Critical problems}\label{SCP}	
In this section, we will study the problem \eqref{e1.1} and focus on the existence and nonexistence of solutions. Throughout this section, we assume that $\Om\subseteq \mb R^n$ is a bounded open set. We inherit all the definitions and notations of Sections \ref{P} and \ref{HLS}.
\begin{Remark}
	Some remarks are in order.
	\begin{enumerate}
		\item Firstly, since $\Om \subseteq \mb R^n$ is bounded, we have
		\begin{align*}
			\Pi(\Om)=\{u \in H^1(\mb R^n): u\arrowvert_\Om \in H_0^1(\Om)~\text{and}~u \equiv 0~\text{a.e. in}~\mb R^n \setminus \Om\}.
		\end{align*}
		As a consequence, the assumption $u \in \Pi(\Om)$ contains the Dirichlet condition $u\equiv 0$ a.e. in $\mb R^n \setminus \Om$.
		\item We also observe that Definition \ref{d4.1} is well-posed, in the sense that all the integrals in \eqref{4.1} are finite. Indeed, if $u, \varphi \in \Pi(\Om)$, we have
		\begin{align*}
		\left|\I\Om  \na u \na \varphi dx +\frac{C(n,s)}{2}\I{\mb R^n}\I{\mb R^n}\right. &\left.\frac{(u(x)-u(y))(\varphi(x)-\varphi(y))}{|x-y|^{N+2s}}dxdy\right|\\
			\leq& \|u\|\|\varphi\|+[u]_s[\varphi]_s \leq 2\mc G(u)\mc G(\varphi )<\infty.
		\end{align*}
		Moreover, since $\Pi(\Om) \hookrightarrow L^{2^*}(\mb R^n)$ and $p <2^*-1$, using classical Hardy-Littlewood-Sobolev and H\"{o}lder's inequality (and taking into account that $u,v=0$ a.e. in $\mb R^n \setminus \Om$) we also have
		\begin{align*}
			\I{\Om}\I{\Om}\frac{|u(x)|^{2_\mu^*}|u(y)|^{2_\mu^*-2}u(y)\varphi(y)}{|x-y|^\mu}dxdy+ \la \I{\Om}u^p\varphi& dx\\
			\leq& C(N,\mu)|u|_{2^*}^2|\varphi|_{2^*}+|\la||\varphi|_{\frac{2^*}{2^*-p}}<\infty. 
		\end{align*}
	\end{enumerate}
\end{Remark}
Before proceeding, we want to introduce the energy functional associated to equation \eqref{e1.1} defined for any $u \in \Pi(\Om)$ by
\begin{equation*}
	J_\la(u)=\frac{1}{2}\mc G(u)^2-\frac{1}{2\cdot2_\mu^*}\I{\Om}\I{\Om}\frac{|u(x)|^{2_\mu^*}|u(y)|^{2_\mu^*}}{|x-y|^\mu}dxdy- \frac{\la}{p+1} \I{\Om}u^{p+1}dx.
\end{equation*}
Then using Mixed Hardy-Littewood-Sobolev inequality \cite{BDVV} and assumptions on $\mu$, $n$ and $p$, we observe that $J_\la \in C^1(\Pi(\Om),\mb R)$ and for any $u, \varphi \in \Pi(\Om)$ one has
\begin{align*}
	\ld J_\la^\prime(u), \varphi \rd =& \I{\mb R^n} \na u\cdot\na \varphi dx + \frac{C(n,s)}{2} \I{\mb R^n}\I{\mb R^n}\frac{(u(x)-u(y))(\varphi(x)-\varphi(y))}{|x-y|^{n+2s}}dxdy\\
	&-\I{\Om}\I{\Om}\frac{|u(x)|^{2_\mu^*}|u(y)|^{2_\mu^*-2}u(y)\varphi(y)}{|x-y|^\mu}dxdy- \la \I{\Om}u^p\varphi dx.
\end{align*}
This shows that $u$ is a weak solution of \eqref{e1.1} if and only if $u$ is a critical point of the functional $J_\la$.\\
With Definition \ref{d4.1} in hand,	we first establish  regularity and nonexistence results of the nonnegative solution of \eqref{e1.1}.	
\subsection{Regularity of solutions and Poho\v{z}aev identity}\label{SRP}
We begin this subsection by establishing the 

\textbf{Proof of Theorem \ref{trr}:}
We start by proving the $L^\infty$ estimate.
	Here we closely follow the proof of \cite[Theorem 6.2]{MS}.  We choose a constant $m>1$ and $\rho>0$ small (an appropriate choice of $\rho$ is given later in the proof) so that for any $x\in \mb R^n$, $w(x):=u(x)/m \in \Pi(\Om)$ satisfies the following:
	\begin{align}\label{e4.11}\nonumber
		\I{\mb R^n} \na w \na \varphi +&\frac{C(n,s)}{2}\I{\mb R^n}\I{\mb R^n}\frac{(w(x)-w(y))(\varphi(x)-\varphi(y))}{|x-y|^{n+2s}}\\
		\leq& \I{\Om}\I{\Om}\frac{|w(y)|^{2_\mu^\ast}|w(x)|^{2_\mu^\ast-2}w(x)\varphi(x)}{|x-y|^\mu}dydx \\
		+ \la \I\Om w^p \varphi dx,			\end{align}
	for every $ 0\leq \varphi \in \Pi(\Om)$ and $|w|_{2_\mu}^\ast =\rho$. Next, for every $ k \in \mb N$, we set \begin{equation*}
		D_k =1-2^{-k}, w_k := w-D_k, ~v_k:=w_k^+:=\max\{v_k,0\} ~\text{and}~\bar{ V}_k := |v_k|_{2_\mu^\ast}.
	\end{equation*}
	Note that using Dominated Convergence theorem, we have
\begin{equation}\label{edc}
	\lim_{k \ra \infty} \bar{V}_k=\left(\I\Om[(w-1)^+]^{2^*}dx\right)^\frac{1}{2^*}.
\end{equation}
We claim that
\begin{equation}\label{0l}
		\lim_{k \ra \infty} \bar{V}_k=0.
\end{equation}
Then combining \eqref{edc} and \eqref{0l}, we will obtain that $u \in L^\infty(\Om)$. Now to prove \eqref{0l}, will show that $\bar{V}_k$ satisfy the following inequality for some $M>0$
\begin{equation}\label{in}
	\bar{V}_k \leq M\theta^{k+1},~\text{with}~ \theta \in (0,1).
\end{equation}
	 To obtain \eqref{in}, we first observe that since $ u\in \Pi(\Om)$ and $\Om$ is bounded, $w_k \in H^1_{loc}(\mb R^n)$. Furthermore, since $u \equiv 0$ a.e. in $\mb R^n \setminus \Om$, we also have
	\begin{equation*}
		w_k =w-D_k =-D_k <0~\text{on}~\mb R^n \setminus \Om
	\end{equation*}
	and thus $v_k=w_k^+ \in \Pi(\Om)$.
	It is a simple observation that $v_k$ satisfies 
	\begin{equation}\label{e4.12}
		|v_k(x)-v_k(y)|^2 \leq (w_k(x)-w_k(y))(v_k(x)-v_k(y))
	\end{equation}	
	and 
	\begin{equation}\label{e4.13}
		\I{\mb R^N} \na w \na v_k=\I{\mb R^N} \na w_k \na v_k = \I{\mb R^n \cap \{u >D_k\} } |\na v_k|^2 =\I{\mb R^n}|\na v_k|^2 dx.
	\end{equation}
	Now for any $k \in \mb N$, $D_{k+1} >D_k$ and so $w_{k+1} <w_k$ a.e. in $\mb R^n$. Also let \begin{equation*}
		C_k := D_{k+1}/(D_{k+1}-D_k) =2^{k+1}-1~\text{for any}~ k \in \mb N.
	\end{equation*}
	We claim that for any $k \in \mb N$
	\begin{equation}\label{e4.14}
		w<C_k v_k ~\text{on}~\{v_{k+1}>0\}.
	\end{equation}
	To see this, let $x \in \{v_{k+1}>0\}$. Then $w(x) >D_{k+1}>D_k$, so $v_k(x)=w(x)-D_k$  and 
	\begin{equation*}
		C_k v_k(x)=w(x)+\frac{D_k}{D_{k+1}-D_k}(w(x)-D_{k+1})>w(x).
	\end{equation*} 
	Notice also that $w_{k+1}(x)-w_{k+1}(y)=w(x)-w(y)$, for any $x, y \in \mb R^n$.
	From this, \eqref{e4.12}, \eqref{e4.13}, \eqref{e4.14} and testing \eqref{e4.11} with $\varphi =v_{k+1}$, we get 
	\begin{align}\nonumber\label{e4.15}
		\mc G(v_{k+1})^2 \leq & \I\Om\I\Om \frac{|w(y)|^{2_\mu^\ast}|w(x)|^{2_\mu^\ast-2}w(x)v_{k+1}(x)}{|x-y|^\mu}dydx +\la \I\Om w^p v_{k+1}dx\\ \nonumber
		=& \I{\{v_{k+1}(x)>0\}}\I\Om \frac{|w(y)|^{2_\mu^\ast}|w(x)|^{2_\mu^\ast-2}w(x)v_{k+1}(x)}{|x-y|^\mu}dydx +\la \I{\{v_{k+1}(x)>0\}}  w^p v_{k+1}dx\\ \nonumber
		\leq& C_k^{2_\mu^\ast-1} \I{\{v_{k+1}(x)>0\}}\I\Om \frac{|w(y)|^{2_\mu^\ast}|v_k(x)|^{2_\mu^\ast-1}v_{k+1}(x)}{|x-y|^\mu}dydx +\la C_k^p \I{\{v_{k+1}(x)>0\}}v_{k}^{p+1}dx\\  \nonumber
		\leq&C_k^{2_\mu^\ast-1} \I{\{v_{k+1}(x)>0\}}\I\Om \frac{|w(y)|^{2_\mu^\ast}|v_k(x)|^{2_\mu^\ast}}{|x-y|^\mu}dydx\\
		&+\la 2^{(k+1)p}|v_k|_{2^\ast}^{p+1}|\{v_{k+1}(x)>0\}|^{\kappa_{p,n}},
	\end{align}
	where $\kappa_{p,n} =\left((1-p)n +2(p+1)\right)/(n-2)$.
	Let us consider the first integral on the right hand side above inequality and we see that
	\begin{align}\nonumber\label{e4.16}
		\I{\{v_{k+1}(x)>0\}}\I\Om \frac{|w(y)|^{2_\mu^\ast}|v_k(x)|^{2_\mu^\ast}}{|x-y|^\mu}dydx
		\leq& \left( \I{\{v_{k+1}(x)>0\}}\I{\{w(y)\geq D_{k+1}\}}+\I{\{v_{k+1}(x)>0\}}\I{\{w(y)< D_{k+1}\}}\right)\\
		& \frac{|w(y)|^{2_\mu^\ast}|v_k(x)|^{2_\mu^\ast}}{|x-y|^\mu}dydx = T_1 +T_2~\text{(say)}.
	\end{align}
	Now using \eqref{e4.14} and classical Hardy-Littlewood-Sobolev inequality, we have
	\begin{align}\nonumber \label{e4.17}
		T_1 =&\I{\{v_{k+1}(x)>0\}}\I{\{w(y)\geq D_{k+1}\}}\frac{|w(y)|^{2_\mu^\ast}|v_k(x)|^{2_\mu^\ast}}{|x-y|^\mu}dydx\\
		\leq& C_k^{2_\mu^\ast}  \I{\{v_{k+1}(x)>0\}}\I{\{w(y)\geq D_{k+1}\}}\frac{|v_k(y)|^{2_\mu^\ast}|v_k(x)|^{2_\mu^\ast}}{|x-y|^\mu}dydx \leq C_k^{2_\mu^\ast} C(n,\mu)|v_k|_{2^*}^{2\cdot2_\mu^\ast}.
	\end{align}
	Next, again using H\"{o}lder's inequality we obtain
	\begin{align}\nonumber\label{e4.18}
		T_2 =& \I{\{v_{k+1}(x)>0\}}\I{\{w(y)< D_{k+1}\}}\frac{|w(y)|^{2_\mu^\ast}|v_k(x)|^{2_\mu^\ast}}{|x-y|^\mu}dydx\\\nonumber
		\leq & D_{k+1}^{2_\mu^\ast} \I{\{v_{k+1}(x)>0\}}|v_k(x)|^{2_\mu^\ast}\I\Om \frac{dy}{|x-y|^\mu}dx\\
		\leq& M D_{k+1}^{2_\mu^\ast} \I{\{v_{k+1}(x)>0\}}|v_k(x)|^{2_\mu^\ast}dx \leq M D_{k+1}^{2_\mu^\ast}|\{v_{k+1}>0\}|^{\frac{\mu}{2n}}|v_k|_{2^\ast}^{2_\mu^\ast}.
	\end{align}
	Using \eqref{e4.16}, \eqref{e4.17}, \eqref{e4.18} and Mixed Sobolev inequality in \eqref{e4.15}, we get
	\begin{align}\nonumber\label{e4.19}
		S|v_{k+1}|_{2^\ast}^2 \leq& \mc G(v_{k+1})^2\\\nonumber
		\leq & C_k^{2_\mu^\ast-1}\left(C_k^{2_\mu^\ast}C(n,\mu)|v_k|_{2^\ast}^{2\cdot2_{\mu}^\ast}+M D_{k+1}^{2_\mu^\ast}|\{v_{k+1}>0\}|^{\frac{\mu}{2n}}|v_k|_{2^\ast}^{2_\mu^\ast}\right.\\
		&\left.+\la 2^{(k+1)p}|v_k|_{2^\ast}^{p+1}|\{v_{k+1}(x)>0\}|^{\kappa_{p,n}}\right).
	\end{align}
	Now it is easy to see that 
	\begin{equation*}
		\{v_{k+1}(x)>0\} \subset \{v_k > 2^{-(k+1)}\}.
	\end{equation*}
	Thus
	\begin{align}\label{e4.20}
		\bar {V}_k^{2^*} =|v_k|_{2^*}^{2^*} \geq \I{\{v_k >2^{-(k+1)}\}}v_k^{2^*} \geq 2^{-2^*(k+1)}|\{v_{k+1}>0\}|.
	\end{align}
	Now if $\la >0$, using \eqref{e4.20}, we conclude from \eqref{e4.19} that
	\begin{align}\nonumber\label{e4.21}
		S|v_{k+1}|_{2^\ast}^2 \leq& C_k^{2_\mu^\ast-1}\left(C_k^{2_\mu^\ast}C(n,\mu)|v_k|_{2^\ast}^{2\cdot2_{\mu}^\ast}+M D_{k+1}^{2_\mu^\ast}2^{\frac{\mu2^*(k+1)}{2n}}|v_k|_{2^\ast}^{2^\ast}\right.\\  \nonumber
		&\left.+\la 2^{(k+1)(p+2^*\kappa_{p,n})}|v_k|_{2^\ast}^{2^\ast}\right)\\ \nonumber
		\leq& 2^{(2_\mu^\ast-1)(k+1)}\left(2^{2_\mu^\ast(k+1)}C(n,\mu)|v_k|_{2^\ast}^{2\cdot2_{\mu}^\ast}+M 2^{\frac{\mu2^*(k+1)}{2n}}|v_k|_{2^\ast}^{2^\ast}\right.\\ \nonumber
		&\left.+\la 2^{(k+1)(p+2^*\kappa_{p,n})}|v_k|_{2^\ast}^{2^\ast}\right)\\ \nonumber
		\leq&2^{(2_\mu^\ast-1)(k+1)} \max\left\{2^{2_\mu^\ast(k+1)}C(n,\mu), M 2^{\frac{\mu2^*(k+1)}{2n}} \right.\\
		&\left. + \la 2^{(k+1)(p+2^*\kappa_{p,n})}\right\}\times \left(|v_k|_{2^\ast}^{2\cdot2_{\mu}^\ast}+|v_k|_{2^\ast}^{2^\ast}\right).
	\end{align}
	Otherwise if $\la\leq 0$, then using \eqref{e4.19} and \eqref{e4.20} we have
	\begin{align}\nonumber\label{e4.22}
		S|v_{n+1}|_{2^\ast}^2 \leq&C_n^{2_\mu^\ast-1}\left(C_n^{2_\mu^\ast}C(N,\mu)|v_n|_{2^\ast}^{2\cdot2_{\mu}^\ast}+M D_{n+1}^{2_\mu^\ast}2^{\frac{\mu2^*(n+1)}{2N}}|v_n|_{2^\ast}^{2^\ast}\right)\\\nonumber
		\leq& 2^{(2_\mu^\ast-1)(n+1)}\left(2^{2_\mu^\ast(n+1)}C(N,\mu)|v_n|_{2^\ast}^{2\cdot2_{\mu}^\ast}+M 2^{\frac{\mu2^*(n+1)}{2N}}|v_n|_{2^\ast}^{2   ^\ast}\right)\\
		\leq&2^{(2_\mu^\ast-1)(n+1)} \max\left\{2^{2_\mu^\ast(n+1)}C(N,\mu), M 2^{\frac{\mu2^*(n+1)}{2N}}\right\}\times \left(|v_n|_{2^\ast}^{2\cdot2_{\mu}^\ast}+|v_n|_{2^\ast}^{2^\ast}\right).
	\end{align}
	Therefore using definition of $\bar{V}_k$ in \eqref{e4.21} and \eqref{e4.22}, we get
	\begin{equation}\label{e4.23}
		\bar{V}_{k+1} \leq  R^{k+1}\left(\bar{V}_k^{2_\mu^\ast}+\bar{V}_k^\frac{2^*}{2}\right),
	\end{equation}
	where
	\begin{equation*}
		R=\begin{cases}
			\left(1+\ds\frac{1}{S}\left(2^{(2_\mu^\ast-1)}\max\left\{2^{2_\mu^\ast}C(n,\mu), M 2^{\frac{\mu}{2n}},2^{(p+\kappa_{p,n})} \right\}\right)^{1/2}\right)~&\text{if}~\la>0,\\
			\left(1+\ds\frac{1}{S}\left(2^{(2_\mu^\ast-1)}\max\left\{2^{2_\mu^\ast}C(n,\mu), M 2^{\frac{\mu}{2n}}\right\}\right)^{1/2}\right)~&\text{if}~\la\leq0.
		\end{cases}
	\end{equation*}
	Note that $R>1$ and also $2_\mu^\ast>1$.\\
	Now we assume $\de >0$ is so small that
	\begin{equation}\label{e4.24}
		\de^{\frac{2^*}{2}-1} < \frac{1}{\left(2^{2_\mu^\ast}R\right)^\frac{1}{(2^*/2)-1}}.
	\end{equation}
	We also fix $\ds\theta \in \left(	\de^{\frac{2^*}{2}-1},\frac{1}{\left(2^{2_\mu^\ast}R\right)^\frac{1}{(2^*/2)-1}} \right)$. Since $R>1$ and $2^{2^*}/2 >1$, we get $\theta \in (0,1)$. Moreover,
	\begin{equation}\label{e4.25}
		\de^{\frac{2^*}{2}-1} \leq \theta~\text{and}~2^{2_\mu^\ast}R\theta^{\frac{2^*}{2}-1}\leq 1.
	\end{equation}
	Now we give our choice of $\rho$, namely we choose $\rho=\de \theta$. We show that $\bar{V}_k$ satisfies \eqref{in} with $M=2\de$.	
	We shall apply mathematical induction to prove our claim. First note that
	\begin{equation*}
		\bar{V}_0 = |w^+|_{2^*}\leq |w|_{2^*} =\rho \leq b\de \leq 2\rho =2\de\theta,
	\end{equation*}
	which is \eqref{in} with $k=1$. Let us now suppose that \eqref{in} is true for $k$ and let us prove it for $k+1$. Using \eqref{e4.23} and \eqref{e4.25}, we have
	\begin{align*}
		\bar{V}_{k+1} \leq & R^{k+1}\left(\bar{V}_k^{2_\mu^\ast}+\bar{V}_k^\frac{2^*}{2}\right)\leq 2^{2_\mu^\ast+1}R^{k+1}(\de\theta^{k+1})^{\frac{2^*}{2}}\\
		\leq& 2\de\left(2^{2_\mu^*}R\theta^{\frac{2^*}{2}-1}\right)^{k+1}\de^{\frac{2^*}{2}-1}\theta^{k+1} \leq 2\de\theta^{k+2}.
	\end{align*}
	This completes the induction and so our claim \eqref{0l} holds and thus $u \in L^\infty(\Om)$. Further since $u\equiv 0$ on $\mb R^n \setminus \Om$, $u \in L^\infty(\mb R^n)$.
	Next we show that $u \in C^{1,\al}(\bar{\Om})$ for some $\al \in (0,1)$. For this, noting that since $0<\mu<n$ and $\Om$ is bounded, we have
	\begin{align*}
		\left|\I\Om\frac{|u(y)|^{2_\mu^\ast}}{|x-y|^\mu}dy\right|\leq& |u|_\infty^{2_\mu^\ast}\left[\I{\Om \cap \{|x-y|<1\}}\frac{dy}{|x-y|^\mu}+\I{\Om \cap \{|x-y|\geq1\}}\frac{dy}{|x-y|^\mu}\right]\\
		\leq& |u|_\infty^{2_\mu^\ast}\left[\I{\Om \cap \{r<1\}}r^{n-1-\mu}+|\Om|\right] <\infty.
	\end{align*} 
Hence the right hand side of \eqref{e1.1} is in $L^\infty(\Om)$. Now if $s \in (0,1/2]$ we use \cite[Theorem 1.4]{SVWZ} and conclude that $u \in W^{2,p}(\Om)$ for every $p \in (1,\infty)$ and so $u \in C^{1,\nu}(\bar{\Om})$ for any $\nu\in (0,1)$. On the other hand if $s \in (1/2,1)$ using \cite[Theorem 2.7]{BDVV4} we conclude that $u \in W^{2,p}(\Om)$ for every $p \in \left(1,\frac{n}{2s-1}\right)$ and so $u \in C^{1,\nu}(\bar{\Om})$ for  $\nu \in (0,2-2s)$. Next according to Riesz potential regularity we show that $v= \ds \I{\Om}\frac{|u(y)|^{2_\mu^*}}{|x-y|^\mu}dy \in C^{0,\beta}(\bar{\Om})$ with $\beta<\min\{s, n-\mu\}$. Let $\gamma \in (0,n)$ such that $\mu =n-\gamma$. Note that since $u \in L^\infty(\mb R^n)$ and $u \equiv 0$ in $\mb R^n \setminus \Om$, we have $|u|^{2_\mu^*}, v \in L^\infty(\Om)$. Using \cite[Proposition 1.4 (iii)]{OS1} (see also \cite{St}), we get $v \in C^{0,\ba}(\bar{\Omega})$ with $\ba <\min\{s,\gamma\}$. Thus we see that right hand side of \eqref{e1.1} is in $C^{0,\ba}(\bar{\Om})$. Now using $C^{2,\al}$ regularity of $\partial \Om$ and applying \cite[Theorem 2.8]{BDVV4}, we conclude that $u \in C^{2,\tilde\ba}(\bar{\Om})$ for $\tilde\ba\leq\min\{\beta, \alpha\}$. \QED
Now we will make use of Theorem \ref{trr} and complete the

\textbf{Proof of Proposition \ref{MP}}:
Suppose that there exists $x_0 \in \Om$ such that $u(x_0)=0$. Then since $u \geq 0$ in $\Om$, $u =\inf_{x\in \bar{\Om}} u(x)$. Then since $u \in C^{2,\al}(\bar{\Om})$, $-\De u(x_0)\leq 0$. Again since $u$ is nontrivial and continuous, $u>0$ on $B_r(y_0)$ for some $y_0 \in \mb R^n$ and $r>0$. From here, it is easy to conclude that $(-\De)^s u(x_0) <0$ and hence left hand side of \eqref{e1.1} is strictly negative at $x_0$. This is a contradiction since the right hand side of \eqref{e1.1} evaluated at $x_0$ is zero. \QED
Next we prove the Poho\v{z}aev identity \eqref{pi} and give the

\textbf{Proof of Proposition \ref{PPI}}:
	Since $u$ is a solution it satisfies \eqref{e1.1} and from Theorem \ref{trr}, $u \in C^{2,\al}(\bar{\Om})$ for some $\al \in (0,1)$. Multiplying \eqref{e1.1} by $(x\cdot \na u)$ and integrating, we get
	\begin{align}\label{0.2}\nonumber
		-\I\Om (x\cdot \na u)\De u +\I\Om(x\cdot \na u)(-\De)^s u =& \I\Om(x\cdot \na u)\left(\ds\I{\Om}\frac{|u(y)|^{2_\mu^\ast}}{|x-y|^\mu}\right)|u|^{2_\mu^\ast-1}dx\\
		& +\la\I\Om (x\cdot \na u) u^{p}dx. 
	\end{align}
Since $u \in C^{2,\al}(\bar{\Om})\cap C^{1,\al}(\bar{\Om})$ for $\al \in (0,1)$
we have
	\begin{align}\label{0.6}
		-\I\Om \De u (\na u \cdot x)dx 
		=\frac{2-n}{2}\I\Om|\na u|^2 -\frac12 \I{\partial \Om}\left(\frac{\partial u}{\partial \nu}\right)^2 \nu(x) \cdot x d\sigma.
	\end{align}
Now for $x \in \bar{\Om}$, define $\de(x)=\text{dist}(x,\partial \Om)$. Then we see that $u/\de \in C^{0,\al}(\bar{\Om})$ for some $\al \in (0,1)$. Indeed, the regularity of $u/\de$ is the same as that of $\ds \frac{\partial u}{\partial \nu}$. Now since $u \in C^{1,\al}(\bar{\Om})$ for some $\al \in (0,1)$ and $\Om$ is $C^{1,1}$,  $\ds \frac{\partial u}{\partial \nu} \in C^{0,\al}(\bar{\Om})$ for some $\al \in (0,1)$ and so is $u/\de$. Thus we are entitled to use
	 Theorems $1.4$ and $1.6$ of \cite{OS} and get
	\begin{align}\label{0.7}
		\I\Om(x\cdot \na u)(-\De)^s u = \frac{2s-n}{2}\I\Om u(-\De)^sudx-\frac{\Gamma(1+s)^2}{2}\I{\partial\Om}\left(\frac{u}{\de^s}\right)^2(\nu(x).x)d\sigma.
	\end{align}
	Using \cite[Proposition 6.2]{GY1}, we also have
	\begin{align}\label{0.8}
		\I\Om (x\cdot \na u(x))\I\Om \frac{|u(y)|^{2_\mu^\ast}}{|x-y|^\mu}dy|u(x)|^{2_\mu^*-1}dx=\frac{\mu-2n}{2.2_\mu^\ast}\I\Om\I\Om\frac{|u(y)|^{2_\mu^\ast}|u(x)|^{2_\mu^\ast}}{|x-y|^\mu}dxdy.
	\end{align}
	Lastly,
	\begin{align}\label{0.9}
		\I\Om (x\cdot \na u)u^pdx =\I\Om x\cdot \na \left(\frac{1}{p+1}u^{p+1}\right)dx =-\frac{n}{p+1}\I\Om u^{p+1}dx.
	\end{align}
	Substituting \eqref{0.6}, \eqref{0.7}, \eqref{0.8} and \eqref{0.9} in \eqref{0.2} \, we obtain \eqref{pi}. \QED
Now we can prove Theorem \ref{t1.4}.\\
\textbf{Proof of Theorem \ref{t1.4}}:
	Let $u$ be a nonnegative weak solution of \eqref{e1.1}. Then $u$ satisfies
\begin{align*}
	\|u\|^2 +[u]_s^2 = \I\Om\I\Om\frac{|u(x)|^{2_\mu^\ast}|u(y)|^{2_\mu^\ast}}{|x-y|^\mu}dxdy+ \la \I\Om |u|^{p+1}dx.
\end{align*}
Using this in \eqref{pi}, we have
\begin{align*}
	(s-1)[u]_s^2-\frac{1}{2}\I{\partial\Om}\left(\frac{\partial u}{\partial \nu}\right)^2\nu(x).x d\sigma-\frac{\Gamma(1+s)^2}{2}&\I{\partial\Om}\left(\frac{u}{\de^s}\right)^2(\nu(x).x)d\sigma\\
	=&-\la\left(	n\left(\frac{1}{p+1}-\frac12\right) +1\right)\I\Om|u|^{p+1}.
\end{align*}
Now using the facts that $\Om$ is strictly star shaped and \eqref{epn}, we have $ u\equiv 0$. This is a contradiction to the fact that $u$ is nontrivial.\QED
Finally we complete this subsection by giving the

\textbf{Proof of Theorem \ref{t4.3}}: Let $s \in (0,3/4)$. We will show that 
 $(-\De)^s u \in C^{0,\ba}(\Om)$ for some $\ba \in (0,1)$. Firstly, if $s \in (0,\frac12]$, using that  $u \in C^{1,\al}(\bar{\Om})$ and $u \equiv 0$ in $\mb R^n \setminus \Om$, we have that $u \in C^{0,1}(\mb R^n)$. Let $K \subset \Om$ be a compact set  and let $\eta$ be a cut-off function, that is, $\eta \in C_0^\infty(\Om)$ such that $\eta(x) \in [0,1]$ for every $x \in \Om$, supp$\eta \subset \Om$, and $\eta(x)=1$ for every $x \in K$ and suppose $v =u \eta$. Since $u\equiv0$ in $\mb R^n \setminus \Om$, we conclude that 
$ v \in  C^{1,\nu}(\mb R^n)$, with any $\nu \in (0,1)$ if $s \leq \frac12$ and any $\nu \in (0,2-2s)$ if $s >\frac12$. 
 Then using \cite[Proposition 2.6-$(ii)$]{S}, we obtain $(-\De)^sv \in C^{0,\ba}(\mb R^n)$ with $\ba \in (0,1)$ if $s \leq \frac12$ and $\ba \in (0,3-4s)$ if $s >\frac12$. Also for any $x \in K$ we have 
\begin{equation*}
	(-\De)^s u(x)= (-\De)^s v(x) +\I{\mb R^n} \frac{u(y)(1-\eta(y ))}{|x-y|^{n+2s}}dy
\end{equation*}
and since the integral on the right hand side of above equation is smooth in $K$, we conclude that $ u \in C^{0,\ba}(K)$ and since the choice of $K$ is arbitral $u \in C_{\text{loc}}^{0,\ba}(\Om)$ with $\ba$ as above. Finally by elliptic regularity theory, we get $ u \in C_{\text{loc}}^{2,\ba}(\Om)$.\\
	Next we prove the Poho\v{z}aev identity \eqref{pi}. First note that \eqref{0.7}, \eqref{0.8} and \eqref{0.9} holds in this case also. Thus we only need  to prove \eqref{0.6}. For this, let $\widetilde{\Om} \subset \subset \Om$ be relatively compact. Then since $u \in C_{\text{loc}}^{2,\al}(\Om)$, we have
		Now using by parts, we have
		\begin{align}\label{0.4}
			-\I{\widetilde{\Om}} \De u (\na u \cdot x)dx =\I{\widetilde{\Om}} \na u\na(\na u\cdot x)dx -\I{\partial {\widetilde{\Om}}} \frac{\partial u}{\partial \nu} \na u(x)\cdot x d\sigma.
		\end{align}
	Since $u \in W^{2,2}(\Om) \cap C^{1,\al}(\bar{\Om})$, we obtain by taking the limit $\widetilde{\Om} \ra \Om$ in \eqref{0.4} and using Dominated convergence theorem that 
		\begin{align}\label{epi}
		-\I{\Om} \De u (\na u \cdot x)dx =\I{\Om} \na u\na(\na u\cdot x)dx -\I{\partial{\Om}} \frac{\partial u}{\partial \nu} \na u(x)\cdot x d\sigma.
	\end{align} 
		We have
		\begin{align*}
			\frac{\partial}{\partial x_j} (\na u \cdot x) =\frac{\partial}{\partial x_j}\left(\sum_{i=1}^{n}\frac{\partial u}{\partial x_i} x_i\right) =\sum_{i=1}^{n} \frac{\partial^2 u}{\partial x_j \partial x_i}x_i +\frac{\partial u}{\partial x_j},
		\end{align*}
		so that
		\begin{align*}
			\na u \cdot \na (\na u \cdot x)=&\sum_{j=1}^{n} \frac{\partial u}{\partial x_j}\left(\sum_{i=1}^{n} \frac{\partial^2 u}{\partial x_j \partial x_i}x_i +\frac{\partial u}{\partial x_j}\right)\\ 
			=& \sum_{i=1}^{n}\frac{\partial }{\partial x_i}\left(\frac12\sum_{j=1}^{n}\left(\frac{\partial u}{\partial x_j}\right)^2 \right)x_i +|\na u|^2\\
			=& \frac12 \na (|\na u|^2)\cdot x +|\na u|^2.
		\end{align*}
		Thus 
		\begin{align}\nonumber\label{0.5}
			\I\Om \na u\na(\na u\cdot x)dx =&\I\Om  \frac12 \na (|\na u|^2)\cdot x +|\na u|^2\\
			=& \I\Om |\na u|^2 -\frac{n}{2} \I\Om |\na u|^2 +\frac12 \I{\partial \Om}|\na u|^2 \nu(x)\cdot x d\sigma.
		\end{align}
		Now we notice that since $u =0$ on $\partial \Om$, we have $\na u(x)=\frac{\partial u}{\partial \nu}(x).\nu(x) $ for every $x \in \partial \Om$, so that $|\na u| = \left|\frac{\partial u}{\partial \nu}\right|$ and $\na u \cdot x =\ds\frac{\partial u}{\partial \nu} \nu(x)\cdot x$ on $\partial \Om$. Taking this into account and inserting \eqref{0.5} in \eqref{epi}, we have \eqref{0.6}.
 \QED
From now on we assume that $\la>0$. Under this assumption and taking into account all the discussion carried out so far, we can start our study of the solvability of problem \eqref{e1.1}. To this end, since the linear case $p=1$ and the superlinear case $p>1$ present some significant differences, we treat these cases separately.
\subsection{The linear case $p=1$}\label{SLC}
 We begin by studying the solvability of \eqref{e1.1} in the linear case, that is, $p=1$.
As we shall see, the existence of solutions to problem \eqref{e1.1} for a given $\la>0$ is related to the first Dirichlet eigenvalues of $(-\De)^s$ and of $ \mc L$, which are simple and we recall their characterization below: 
\begin{Definition}
	Let $\Om$ be a bounded open set. We define
	\begin{enumerate}
		\item the first Dirichlet eigenvalue of $(-\De)^s$ in $\Om$ as follows
		\begin{equation}\label{el1s}
			\la_{1,s}:=\inf\{[u]_s^2: u \in C_0^\infty(\Om) ~\text{and}~|u|_2=1\};
		\end{equation}
		\item the first Dirichlet eigenvalue of $\mc L$ in $\Om$ as
		\begin{equation}\label{el1}
			\la_1 := \inf\{ \mc G(u)^2: u \in C_0^\infty(\Om) ~\text{and}~|u|_2=1\}.
		\end{equation}
	\end{enumerate}
\end{Definition}
Before starting the proof of Theorem \ref{tlc}, we recall (for the sake of completeness) the main properties of $\la_{1,s}$ and $\la_1$.
\begin{Remark}\label{r0.3}
	We recall that since $\Om$ is bounded both $\la_{1,s}$ and $\la_1$ are achieved in the space $\Pi(\Om)$. This means 
			\begin{equation*}
			\text{there exists}~v_0 \in \Pi(\Om) :|v_0|_2=1~\text{and}~[v_0]_s^2=\la_{1,s}>0,
		\end{equation*}
		\begin{equation*}
		\text{ and there exists}~w_0 \in \Pi(\Om) :|w_0|_2=1~\text{and}~\mc G(w_0)^2=\la_{1}>0.
	\end{equation*}
Also the above functions $v_0$ and $w_0$ can be chosen to be nonnegative and then using regularity results and maximum principle, we have $v_0, w_0$ are strictly positive in $\Om$. Lastly, since $v_0$ and $w_0$ are constrained minimizers of $u \mapsto [u]_s^2$ and $u \mapsto \mc G(u)^2$ respectively, by the Lagrange Multiplier rule we easily see that 
\begin{equation*}
	\I{\mb R^n}\I{\mb R^n}\frac{(v_0(x)-v_0(y))(\varphi(x)-\varphi(y))}{|x-y|^{N+2s}}dxdy =\la_{1,s}\I\Om v_0\varphi~\text{for every}~\varphi \in \Pi(\Om),
\end{equation*}
	\begin{align*}
\text{and}~	\I\Om \na w_0\na \varphi dx + \I{\mb R^n}\I{\mb R^n} \frac{(w_0(x)-w_0(y))(\varphi(x)-\varphi(y))}{|x-y|^{N+2s}}&dxdy\\
	=& \la_1 \I\Om w_0 \varphi dx ~\text{for all}~\varphi \in \Pi(\Om).
\end{align*}
\end{Remark}
We now begin the proof of Theorem \ref{tlc}. We shall give the proof through several independent results. To begin with, we prove a lemma linking the existence of solutions to \eqref{e1.1} with the existence of constrained minimizer for a suitable functional. To this end, for any $u \in \Pi(\Om)$, we define
\begin{equation}
	\mc P_\la (u):= \mc G(u)^2-\la |u|_2^2.
\end{equation}
We also define the manifold $\mc V(\Om):= \Pi(\Om)\cap \mc H (\Om)$, where
\begin{equation*}
	\mc H(\Om):=\{u \in L^{2^*}(\mb R^n): \|u\|_{HL}=1\}.
\end{equation*}	
Note that the above set is well defined by Hardy-littlewood Inequality (see Proposition \ref{hls}). Then we have the following lemma.
\begin{Lemma}\label{le4.7}
	For every $\la >0$, we define
	\begin{align*}
		S_{H,L}(\la):= \inf_{u \in \mc V (\Om)} \mc P_\la(u). 
	\end{align*}
	We assume that $S_{H,L}(\la)>0$ and that $S_{H,L}(\la)$ is achieved, that is, there exists some function $ \psi \in \mc V(\Om)$ such that $\mc P_\la(\psi)=S_{H,L}(\la)$.\\
	Then, there exists a solution to \eqref{e1.1}.
\end{Lemma}	  
Before giving the proof of Lemma \ref{le4.7}, we list in the next remark some properties of the number $S_{H,L}(\la)$ which will be used later.
\begin{Remark}\label{r4.8}
	From the definition of $\mc S_{H,L}(\la)$ we easily infer that
	\begin{enumerate}[label=(\roman*)]
		\item $S_{H,L}(\la) \leq S_{H,L,C}:=\inf\limits_{u \in C_0^\infty({\mb R}^n)\setminus \{0\}} \frac{\| u\|^2}{\left(\I{\mb R^n}\I{\mb R^n}\frac{|u(x)|^{2_\mu^\ast}|u(y)|^{2_\mu^\ast}}{|x-y|^\mu}dxdy\right)^\frac{1}{2_\mu^\ast}}$ for every $\la >0$;
		\item $S_{H,L}(\la) \leq S_{H,L}(\nu)$ for every $0<\nu\leq\la$.
	\end{enumerate}
	Also, owing to the definition of $\la_1$ in \eqref{el1} (and recalling that $\la_1$ is achieved in the space $\Pi(\Om)$), it is easy to see that
	\begin{equation*}
		S_{H,L}(\la) \geq 0 \iff 0<\la \leq \la_1.
	\end{equation*}
\end{Remark}
We are now ready to give the proof of Lemma \ref{le4.7}.

\textbf{Proof of Lemma \ref{le4.7}}
From assumptions, we know that there exists $\psi \in \mc V(\Om)$ a constrained minimizer for the functional $\mc P_\la$, that is,
\begin{equation*}
	\mc P_\la (\psi)=S_{H,L}(\la).
\end{equation*}
Since $\mc P_\la(|\psi|) \leq \mc P_\la(\psi)$, without loss of generality we assume that $\psi \geq 0$ a.e. in $\Om$. Now using the Lagrange Multiplier Rule, there exists $\nu \in \mb R$ such that
\begin{align}\nonumber\label{4.24}
	\I\Om\na \psi \na \varphi + &\I{\mb R^n}\I{\mb R^n} \frac{(\psi(x)-\psi(y))(\varphi(x)-\varphi(y))}{|x-y|^{n+2s}}dxdy\\
	=&\nu\I\Om\I\Om\frac{|\psi(y)|^{2_\mu^*}|\psi(x)|^{2_\mu^*-2}\psi(x)\varphi(x)}{|x-y|^\mu}dxdy + \la \I\Om \psi \varphi~\text{for all}~\varphi \in \Pi(\Om).
\end{align}
Using $\psi$ as test function in \eqref{4.24}, we get
\begin{align*}
	\nu = \nu \|\psi\|_{HL}^{2_\mu^\ast}=\mc G(\psi)^2-\la\|\psi\|_2^2=\mc  P_\la(\psi)=S_{H,L}(\la)>0.
\end{align*}
As a consequence, setting $u =S_{H,L}(\la)^{\frac{(n-2)}{2n+4-2\mu}}\psi$, we see that $u \geq 0$ a.e. in $\Om$ and for every $\varphi \in \Pi(\Om)$ using \eqref{4.24} we have
\begin{align*}
	\I\Om\na u \na \varphi + \I{\mb R^n}\I{\mb R^n} \frac{(u(x)-u(y))(\varphi(x)-\varphi(y))}{|x-y|^{n+2s}}&dxdy\\
	=&\I\Om\I\Om\frac{|u(y)|^{2_\mu^*}|u(x)|^{2_\mu^*-2}u(x)\varphi(x)}{|x-y|^\mu}dxdy + \la \I\Om u \varphi.
\end{align*}
This completes the proof. \QED
As we see in Lemma \ref{le4.7}, sign of the real number $S_{H,L}(\la)$ plays an important role to study the solvability of problem \eqref{e1.1}. The following result combined with Remark \ref{r4.8} provide further informations in this regard:
\begin{Lemma}\label{le4.9}
	For every $0<\la \leq \la_{1,s}$, we have
	\begin{equation*}
		S_{H,L}(\la)=S_{H,L,C}>0.
	\end{equation*}
\end{Lemma}
\begin{proof}
	Let $\la \in (0,\la_{1,s})$ be given. Using Remark \ref{r4.8}, we already know that $	S_{H,L}(\la)\leq S_{H,L,C}$. Now to prove the reverse inequality, using the definition of $\la_{1,s}$ in \eqref{el1s}, for any $u \in C^\infty_0(\Om)\cap \mc H(\Om)$, we have
	\begin{align*}
		\mc P_\la(u)= \|u\|^2+([u]_s^2-\la|u|_2^2)\geq \|u\|^2+(\la_{1,s}-\la)|u|_2^2 \geq \|u\|^2.
	\end{align*}
	As a consequence, since $C_0^\infty(\Om)$ is dense in $\Pi(\Om)$, we obtain
	\begin{align*}
		S_{H,L}(\la) =&\inf\{\mc P_\la(u): u \in C^\infty_0(\Om)\cap \mc H(\Om)\}\\
		\geq& \inf\{\|u\|^2: u \in C^\infty_0(\Om)\cap \mc H(\Om) \} =S_{H,L,C}.
	\end{align*}
	Thus we get the desired result. \QED
\end{proof}
By combining Lemma \ref{le4.9} and Remark \ref{r4.8} we have the following
\begin{enumerate}
	\item $S_{H,L}(\la) =S_{H,L,C}$ for every $0<\la \leq \la_{1,s}$;
	\item $S_{H,L}(\la) \geq 0$ for every $0<\la \leq \la_1$;
	\item $S_{H,L}(\la) <0$ for every $\la >\la_1$.
\end{enumerate}

\begin{Lemma}\label{l0.7}
	The mapping $\la \mapsto S_{H,L}(\la)$ is continuous from left on $(0,\infty)$.
\end{Lemma}
\begin{proof}
	To prove the left continuity, let $\la_0>0$ and $ \e >0$ be given. By definition of $S_{H,L}(\la_0)$ there exists $v=v_{\e,\la_0} \in \mc V(\Om)$ such that 
	\begin{align*}
		S_{H,L}(\la_0) \leq \mc P_{\la_0}(v)<S_{H,L}(\la_0)+\frac{\e}{2}.
	\end{align*}
	From this, using the monotonicity of $S_{H,L}(\cdot)$, for every $\la <\la_0$ we obtain
	\begin{align*}
		0<S_{H,L}(\la)-S_{H,L}(\la_0) \leq& \mc P_\la(v)-S_{H,L}(\la_0)
		= (\mc P_{\la_0}(v)-S_{H,L}(\la_0))+(\la_0-\la)|v|_2^2\\
		<& \frac{\e}{2}+(\la_0-\la)|v|_2^2.
	\end{align*}
	As a consequence, setting $\de_\e:=\e/(2|v|_2^2)$, we conclude that
	\begin{equation*}
		0<S_{H,L}(\la)-S_{H,L}(\la_0)<\e ~\text{for every}~\la_0-\de_\e<\la\leq \la_0.
	\end{equation*}
	 This proves that $S_{H,L}(\cdot)$ is continuous from the left at $\la_0$.\QED
\end{proof}
In the light of results obtained so far, we can prove Theorem \ref{tlc}.\\
\textbf{Proof of Theorem \ref{tlc}}:
	To begin with, we define
	\begin{equation*}
		\la^* := \sup\{\la>0:S_{H,L}(\mu) =S_{H,L,C}~\text{for all}~0<\mu<\la\}.
	\end{equation*}
	On account of Lemma \ref{le4.9}, we see that $\la_{1,s}\leq \la^*$. Again since $S_{H,L}(\la_1)=0$ (see Remark \ref{r0.3}) and using Lemma \ref{l0.7}, we conclude that $\la^* \in [\la_{1,s},\la_1)$. We treat the following three cases separately.\\
	\textbf{Case 1}: $0<\la\leq \la_{1,s}$. Arguing by contradiction, let us suppose there exists a solution to problem \eqref{e1.1} such that $u \in \mc B$. Now set $w
	=\ds \frac{u}{\|u\|_{HL}}$. Then we have
	\begin{align*}
		\mc P_\la(w) =\frac{1}{\|u\|^2_{HL}}\mc P_\la(u)=\frac{1}{\|u\|^2_{HL}}(\mc G(u)^2-\la|u|_2^2)=\|u\|_{HL}^{2_\mu^\ast-2}\leq S_{H,L,C}.
	\end{align*}
	As a consequence of Lemma \ref{le4.9}, we obtain	
	\begin{equation*}
		\mc P_\la(w)\leq S_{H,L,C}=S_{H,L}(\la).
	\end{equation*}
	This implies that $S_{H,L}(\la)$ is achieved at $w$. Now using $\la \leq \la_{1,s}$, we have
	\begin{align*}
		S_{H,L,C} \leq \|w\|^2 =\mc P_\la(w)-([w]_s^2-\la |w|_2^2)\leq\mc P_\la(w)=S_{H,L,C}
	\end{align*}
	which shows that $S_{H,L,C}$ is achieved. \\
	On the other hand, since $\Om$ is a bounded domain we know that $S_{H,L,C}$ is never achieved and so we have a contradiction. \\
	\textbf{Case 2}:  $\la*<\la <\la_1$. Then  owing to the definition of $\la^*$, we can find $0<\mu<\la$ such that $S_{H,L}(\mu)<S_{H,L,C}$ and so from Remark \ref{r4.8}-$(ii)$ $S_{H,L}(\la)<S_{H,L,C}$. Then noting that Brezis Lieb Lemma type hold for $\|\cdot\|_{HL}$ (see \cite[Lemma 2.2]{GY1}) and following proof of \cite[Lemma 1.2]{BN} we see that $S_{H,L}(\la)$ is achieved.
	Next we show that $S_{H,L}(\la)>0$. Now after renormalization of $u$, we have $S_{H,L}(\la)=\mc P_\la(u)$. In particular, since $\la <\la_1$, we obtain
	\begin{equation*}
		S_{H,L}(\la) =|u|_2^2 \left[\frac{\mc G(u)^2}{|u|_2^2}-\la\right] \geq |u|_2^2(\la_1-\la)>0.
	\end{equation*}
	Now we have proved that $S_{H,L}(\la)$ is strictly positive and is achieved, we are then entitled to apply Lemma \ref{le4.7}, ensuring there exists a solution to \eqref{e1.1}.\\
	\textbf{Case 3}: $\la \geq \la_1$.	Arguing by contradiction, let us assume that there exists a solution to \eqref{e1.1}. Then by Maximum principle (see Proposition \ref{MP}), we have $u>0$ a.e. in $\Om$. By Remark \ref{r0.3} we know there exists $w_0 \in \Pi(\Om)$ such that $w_0 >0$ a.e. in $\Om$ and 
	\begin{align*}
		\I\Om \na w_0\na \varphi dx + \I{\mb R^n}\I{\mb R^n} \frac{(w_0(x)-w_0(y))(\varphi(x)-\varphi(y))}{|x-y|^{N+2s}}&dxdy\\
		=& \la_1 \I\Om w_0 \varphi dx ~\text{for all}~\varphi \in \Pi(\Om).
	\end{align*}
	Taking $\varphi =u$ above, we get
	\begin{align*}
		\la_1 \I\Om w_0 u =&\I\Om \na w_0\na \varphi dx + \I{\mb R^n}\I{\mb R^n} \frac{(w_0(x)-w_0(y))(\varphi(x)-\varphi(y))}{|x-y|^{N+2s}}\\
		=& \I{\Om}\I{\Om}\frac{|w_0(y)|^{2_\mu^\ast}|w_0(x)|^{2_\mu^\ast-2}w_0(x)u(x)}{|x-y|^\mu}dydx+\la\I\Om u w_0dx\\
		>&\la \I\Om u w_0dx
	\end{align*}	
	but this is a contradiction to $\la >\la_1$. This concludes the proof. \QED
\subsection{ The superlinear case} \label{SSC}
In this subsection we study the case $1 <p < 2^*-1$ and prove Theorem \ref{t4.1}.	
	We will employ Mountain Pass Theorem to obtain the existence of a solution. As in the purely local case (see \cite{GY}) the main difficulty to apply Mountain Pass Theorem consists in proving the validity of a $(PS)_c$ condition at a level $c \in \mb R$. Precisely, we have to prove that the Palais Smale condition holds for any $c$ strictly below the Mountain Pass first critical level (given in Lemma \eqref{l4.4}).
 
	 In the next lemma, we show that $J_\la$ possesses the Mountain Pass geometry.
	\begin{Lemma}\label{l4.2}
		If $1<p<2^*-1$ and $\la >0$, then the functional $J_\la$ satisfies the following properties:
		\begin{enumerate}
			\item There exists $\al$, $\sigma >0$ such that $J_\la(u)\geq \al$ for $\mc G(u)=\sigma$.
			\item There exists $e \in \Pi(\Om)$ with $\mc G(e)>\sigma$ such that $J_\la(e)<0$.
		\end{enumerate}
	\end{Lemma}
\begin{proof}
	The proof is standard and is thus omitted. \QED
	\end{proof}
\begin{Lemma}\label{l4.3}
	Let $1<p<2^*-1$, $\la>0$. If $\{u_k\}_{k\in \mb N}$ is a $(PS)_c$ sequence of $J_\la$, then $\{u_k\}$ is bounded. Let $u_0 \in \Pi(\Om)$ be the weak limit of $\{u_k\}_{k \in \mb N}$, then $u_0$ is a weak solution of problem \eqref{e1.1}. Moreover, $J_\la(u_0) \geq 0$.
\end{Lemma}
\begin{proof}First note that $J_\la(|u|) \leq J_\la(u)$, $u \in \Pi(\Om)$, without loss generality we assume that $\{u_k\}_{k \in \mb N}$ is a sequence of nonnegative functions.
	Since $\{u_k\}_{k \in \mb N}$ is a $(PS)_c$ sequence, we can easily find $M >0$ such that 
	\begin{equation*}
		|J_\la(u_k)| \leq M,~|\ld J_\la^\prime(u_k),u_k/\mc G(u_k)\rd | \leq M.
	\end{equation*}
  From here it is easy standard to show that $\{u_k\}$ is bounded in $\Pi(\Om)$ (see for instance \cite[Lemma 2.2]{GY}).
Since $\Pi(\Om)$ is a Hilbert space, we can assume that up to a subsequence, there exists $v \in \Pi(\Om)$, $v \geq 0$ a.e. in $\Om$ such $ u_k \rightharpoonup v$ weakly in $\Pi(\Om)$ and $u_k \ra v$ a.e. in $\Om$. 
Rest of the proof follows similarly as the proof of \cite[Lemma 2.2]{GY}. \QED 
\end{proof}
In the next lemma, we give the threshold value below which the energy functional $J_\la^+$ satisfies the $(PS)_c$ condition. This will play an important role in applying the critical point theorems.
\begin{Lemma}\label{l4.4}
	Assume that $1<q<2^*-1$ and $\la >0$. If $\{u_k\}_{k \in \mb N}$ is a $(PS)_c$ sequence of $J_\la$ with 
	\begin{equation}
		c < \frac{n+2-\mu}{4n-2\mu} S_{H,L,C}^\frac{2n-\mu}{n+2-\mu},
	\end{equation}
then $\{u_k\}_{k \in \mb N}$ has a convergent subsequence.
\end{Lemma}
\begin{proof}
	The proof is similar as the proof of \cite[Lemma 2.4]{GY}. \QED
	\end{proof}
\begin{Lemma}\label{l4.5}
	There exists $w_\e$ such that  
	\begin{equation}
		\sup\limits_{s \geq 0} J_\la(sw_\e) < 	\frac{n+2-\mu}{4n-2\mu}S_{H,L,C}^\frac{2n-\mu}{n-\mu+2}
	\end{equation}
provided that either
\begin{enumerate}
	\item $n> \max\left\{\min\left\{\frac{2(p+3)}{p+1},2+\frac{\mu}{p+1},2\left(1+\frac{2-2s}{p-1}\right)\right\},\frac{2(p+1)}{p}\right\}$ and $\la >0$, or
	\item $n\leq \max\left\{\min\left\{\frac{2(p+3)}{p+1},2+\frac{\mu}{p+1},2\left(1+\frac{2-2s}{p-1}\right)\right\},\frac{2(p+1)}{p}\right\}$ and $\la$ is sufficiently large.
\end{enumerate}
\end{Lemma}
\begin{proof}
	Let $\ds V(x)=\frac{[n(n-2)]^\frac{n-2}{4}}{(1+|x|^2)^\frac{n-2}{2}}$. Then we know from \cite{GY1} that $V$ is a minimizer of $S_{H,L,C}$. Assume that $B_\de \subset \Om \subset B_{2\de}$ and let $\eta \in C_0^\infty(\Om)$ be such that $0\leq \eta \leq 1$, $\eta(x)=1$ in $B_\de$ and $\eta(x)=0$ in $\mb R^n \setminus \Om$. We define, for $\e >0$,
	\begin{equation}\label{e4.6}
		V_\e(x):=\e^\frac{2-n}{2}V\left(\frac{x}{\e}\right)~\text{and}~v_\e(x):= \eta(x)V_\e(x).
	\end{equation}
Then from \cite[p22]{BDVV} and \cite[Lemma 2.5]{GY}, we know that as $\epsilon\to 0$
\begin{equation}\label{e4.7}
	\|v_\e\|^2 =C(n,\mu)^{\frac{n-2}{2n-\mu}\cdot \frac{n}{2}}S_{H,L,C}^\frac{n}{2}+O(\e^{n-2}),
\end{equation}	
\begin{equation}\label{e4.8}
\I\Om\I\Om \frac{|v_\e(x)|^{2_\mu^\ast}|v_\e(x)|^{2_\mu^\ast}}{|x-y|^\mu} \geq C(n,\mu)^\frac{n}{2}S_{H,L,C}^\frac{2n-\mu}{2}-O(\e^{n-\frac{\mu}{2}}),
\end{equation}
and defining $\nu_{s,n}=\min\{n-2,2-2s\}$
\begin{equation}\label{e4.9}
	[v_\e]_s =O(\e^{\nu_{s,n}}).
\end{equation}
Now we argue as in \cite{GY} and consider the following two cases:\\
\textbf{Case 1.} $n> \max\left\{\min\left\{\frac{2(p+3)}{p+1},2+\frac{\mu}{p+1},2\left(1+\frac{2-2s}{p-1}\right)\right\},\frac{2(p+1)}{p}\right\}.$\\
First by the proof of \cite[Lemma 4.1]{DH}, since $p>1$ and $n> \frac{2(p+1)}{p}$, we know $n<(n-2)(p+1)$ and then as $\e\to 0$
\begin{equation}\label{e4.10}
	|v_\e|_{p+1}^{p+1}=O(\e^{n-\frac{(n-2)(p+1)}{2}})+O(\e^\frac{(n-2)(p+1)}{2})=O(\e^{n-\frac{(n-2)(p+1)}{2}}).
\end{equation}
Using the estimates in \eqref{e4.7}, \eqref{e4.8}, \eqref{e4.9} and \eqref{e4.10}, we have
\begin{align*}
	J_\la(sv_\e)=&\frac{s^2}{2}\mc G(v_\e)^2 -\frac{\la s^{p+1}}{p+1}|v_\e|_{p+1}^{p+1}-\frac{s^{2\cdot2_\mu^\ast}}{2\cdot2_\mu^\ast}\|v_\e\|_{HL}\\
	\leq & \frac{s^2}{2}\left(C(n,\mu)^{\frac{n-2}{2n-\mu}\cdot\frac{n}{2}}S_{H,L,C}^\frac{n}{2}+O(\e^{\nu_{s,n}})\right) -\frac{\la s^{p+1}}{p+1}O(\e^{n-\frac{(n-2)(p+1)}{2}})\\
	&-\frac{s^{2\cdot2_\mu^\ast}}{2\cdot2_\mu^\ast}\left(C(n,\mu)^\frac{n}{2}S_{H,L,C}^\frac{n}{2}-O(\e^{n-\frac{\mu}{2}})\right)\\
	:=&f(s).
\end{align*}
Note that $f(s) \ra -\infty$ as $s \ra \infty$ and $f$ is increasing near $0$. Thus there exists $s_\e >0$ such that $\sup\limits_{s>0}f(s)$ is achieved at some  $s_\e$. Then clearly $s_\e$ satisfies the following
\begin{equation*}
	s_\e < \left(\frac{C(n,\mu)^{\frac{n-2}{2n-\mu}\cdot\frac{n}{2}}S_{H,L,C}^\frac{n}{2}+O(\e^{\nu_{s,n}})}{C(n,\mu)^\frac{n}{2}S_{H,L,C}^\frac{2n-\mu}{2}-O(\e^{n-\frac{\mu}{2}})}\right)^\frac{1}{2\cdot 2_\mu^\ast-2}:= S_{H,L,C}(\e)
\end{equation*}
and it is easy to see that there exists $s_0>0$ independent of $\e$ such that for all $\e>0$ small enough $s_\e >s_0$.
Notice that the function 
\begin{equation*}
	s \ra \frac{s^2}{2}\left(C(n,\mu)^{\frac{n-2}{2n-\mu}\cdot\frac{n}{2}}S_{H,L,C}^\frac{2n-\mu}{2}+O(\e^{\nu_{s,n}})\right)-\frac{s^{2\cdot2_\mu^\ast}}{2\cdot2_\mu^\ast}\left(C(n,\mu)^\frac{n}{2}S_{H,L,C}^\frac{2n-\mu}{2}-O(\e^{n-\frac{\mu}{2}})\right)
\end{equation*}
is increasing on $[0,S_{H,L,C}(\e)]$, we have
\begin{align*}
	\max\limits_{s \geq 0} J_\la(sv_\e) \leq& f(s_\e)\\
	=&\frac{s_\e^2}{2}\left(C(n,\mu)^{\frac{n-2}{2n-\mu}\cdot\frac{n}{2}}S_{H,L,C}^\frac{n}{2}+O(\e^{\nu_{s,n}})\right) 
	-\frac{s_\e^{2\cdot2_\mu^\ast}}{2\cdot2_\mu^\ast}\left(C(n,\mu)^\frac{n}{2}S_{H,L,C}^\frac{n}{2}-O(\e^{n-\frac{\mu}{2}})\right)\\
	&-O(\e^{n-\frac{(n-2)(p+1)}{2}})\\
	\leq&\frac{S^2_{H,L,C}(\e)}{2}\left(C(n,\mu)^{\frac{n-2}{2n-\mu}\cdot\frac{n}{2}}S_{H,L,C}^\frac{n}{2}+O(\e^{\nu_{s,n}})\right) \\
	&-\frac{S^{2\cdot2_\mu^\ast}_{H,L,C}(\e)}{2\cdot2_\mu^\ast}\left(C(n,\mu)^\frac{n}{2}S_{H,L,C}^\frac{n}{2}-O(\e^{n-\frac{\mu}{2}})\right)
	-O(\e^{n-\frac{(n-2)(p+1)}{2}})\\
		\end{align*}
\begin{align*}
	 =&\frac{n+2-\mu}{4n-2\mu} \left(\frac{C(n,\mu)^{\frac{n-2}{2n-\mu}\cdot\frac{n}{2}}S_{H,L,C}^\frac{n}{2}+O(\e^{\nu_{s,n}})}{\left(C(n,\mu)^\frac{n}{2}S_{H,L,C}^\frac{2n-\mu}{2}-O(\e^{n-\frac{\mu}{2}})\right)^\frac{n-2}{2n-\mu}}\right)^\frac{2n-\mu}{n+2-\mu}-O(\e^{n-\frac{(n-2)(p+1)}{2}})\\
	\leq& \frac{n+2-\mu}{4n-2\mu}S_{H,L,C}^\frac{2n-\mu}{n+2-\mu} +O(\e^{\min\{\nu_{s,n},n-\frac{\mu}{2}\}})-O(\e^{n-\frac{(n-2)(p+1)}{2}})\\
	\leq&\frac{n+2-\mu}{4n-2\mu}S_{H,L,C}^\frac{2n-\mu}{n+2-\mu},
\end{align*}
since $s_0 <s_\e<S_{H,L,C}(\e)$, \eqref{e4.10} and $n >\min\left\{\frac{2(p+3)}{p+1},2+\frac{\mu}{p+1},2\left(1+\frac{2-2s}{p-1}\right)\right\}$.\\
\textbf{Case 2.}  $n \leq \max\left\{\min\left\{\frac{2(p+3)}{p+1},2+\frac{\mu}{p+1},2\left(1+\frac{2-2s}{p-1}\right)\right\},\frac{2(p+1)}{p}\right\}.$\\
Again since for any fixed $\e$ in \eqref{e4.6} $\ds J_\la(s v_\e) \ra -\infty$ as $s \ra +\infty$, we have that $\ds \max\limits_{s \geq 0} J_\la (sv_\e)$ is achieved at some $s_{\lambda,\e}$ and $s_{\la,\e}$ satisfies
\begin{equation*}
	s_{\la,\e} \mc G(v_\e)^2 =\la s_{\la,\e}^p |v_\e|_{p+1}^{p+1}+s_{\la,\e}^{2\cdot2_\mu^\ast-1}\|v_\e\|^{2\cdot2_\mu^\ast}_{HL},
\end{equation*}
which implies
\begin{equation*}
	\mc G(v_\e)^2=\la s_{\la,\e}^{p-1} |v_\e|_{p+1}^{p+1}+s_{\la,\e}^{2\cdot2_\mu^\ast-2}\|v_\e\|^{2\cdot2_\mu^\ast}_{HL}.
\end{equation*}
Thus $s_{\la,\e} \ra 0$ as $\la \ra \infty$. Then
\begin{align*}
	\max\limits_{s\geq0} J_\la(sv_\e)=\frac{s_{\la,\e}^2}{2}\mc G(v_\e)^2 -\frac{\la s_{\la,\e}^{p+1}}{p+1}|v_\e|_{p+1}^{p+1}-\frac{s_{\la,\e}^{2\cdot2_\mu^\ast}}{2\cdot2_\mu^\ast}\|v_\e\|^{2\cdot2_\mu^\ast}_{HL} \ra 0
\end{align*}
as $\la \ra \infty$, which easily yields the desired conclusion for this case. \QED
	\end{proof} 
\textbf{Proof of Theorem \ref{t4.1}}
Combining Lemmata  \ref{l4.2}, \ref{l4.3}, \ref{l4.4}, \ref{l4.5} and using Mountain Pass Theorem without the $(PS)$ condition, we conclude that $J^+_\la$ has a critical value $c \in  \ds \left(0,\frac{n+2-\mu}{4n-2\mu}S_{H,L,C}^\frac{2n-\mu}{n-\mu+2}\right)$ and thus the problem \eqref{e1.1} has a nontrivial solution. \QED 

	\textbf{Acknowledgement:} The first author thanks the CSIR(India) for financial support in the form of a Senior Research Fellowship, Grant Number $09/086(1406)/2019$-EMR-I. The second author is partially funded by IFCAM (Indo-French Centre for Applied Mathematics) IRL CNRS 3494.

		\end{document}